\documentclass[onecolumn]{article}
\usepackage[margin=1in]{geometry}
\usepackage[dvipdfmx]{graphicx,color}
\usepackage{CJKutf8}
\usepackage{amsthm}
\usepackage{amsfonts, amssymb, bm}
\usepackage{mathtools}
\usepackage{tcolorbox}
\usepackage{subfig}
\usepackage{ascmac}
\usepackage{graphicx}
\usepackage{comment}
\usepackage{mathrsfs}
\usepackage{amscd}
\usepackage[all]{xy}
\usepackage{bbm}
\usepackage{tikz-cd}
\usepackage{array}
\usepackage{amsmath}
\usepackage{tikz}
\usepackage{cancel}
\usepackage{color}
\usepackage{here}
\usepackage{listings}
\usepackage{float}
\usepackage{verbatim}
\usepackage{lipsum} 
\usepackage{multicol} 
\usepackage{xcolor}
\usepackage{colortbl}
\usepackage{longtable}
\usepackage{multirow}
\usepackage{amssymb}
\usepackage{booktabs}
\usepackage{pdflscape}
\usepackage{graphicx} 
\usepackage{empheq}
\numberwithin{equation}{section}

\usepackage[dvipdfmx]{graphicx}
\usetikzlibrary{calc}

\tcbuselibrary{breakable}

\makeatletter
\newcommand{\address}[1]{\gdef\@address{#1}}
\newcommand{\printaddress}{\@address}
\makeatother
\newtheoremstyle{mytheoremstyle}
  {10pt}
  {10pt}
  {\itshape}
  {}
  {\bfseries}
  {}
  {1em}
  {\thmname{#1}\thmnumber{ #2}\thmnote{ (#3)}}

\theoremstyle{mytheoremstyle}

\newtheorem{thm}{Theorem}[section]
\newtheorem{prop}[thm]{Proposition}
\newtheorem{defi}[thm]{Definition}

\newtheorem{remark}[thm]{Remark}
\newtheorem{lem}[thm]{Lemma}

\makeatletter
\renewenvironment{proof}[1][\proofname]{%
  \par\normalfont 
  \topsep6pt\relax 
  \trivlist\item[\hskip\labelsep\itshape #1\@addpunct{.}]\ignorespaces
}{%
  \qed\endtrivlist\@endpefalse
}
\makeatother

\DeclareMathOperator{\SL}{\mathrm{SL}}
\DeclareMathOperator{\PSL}{\mathrm{PSL}}

\DeclareMathOperator{\GF}{\mathrm{GF}}

\DeclareMathOperator{\Orbit}{\mathrm{Orbit}}

\makeatletter
\newenvironment{mmatrix}{%
  \left(\renewcommand\arraystretch{1.5}\begin{smallmatrix}%
}{%
  \end{smallmatrix}\right)}
\makeatother
\begin{document}
\title{Simple $3$-designs of $\PSL(2,2^{n})$ with block size $13$}

\author{
Takara Kondo \quad Yuto Nogata\\
\normalsize{Graduate School of Science and Technology, Hirosaki University}\\
\normalsize{3 Bunkyo-cho, Hirosaki, Aomori, 036-8560, Japan}\\
\normalsize{Email: h23ms107@hirosaki-u.ac.jp \quad h24ms113@hirosaki-u.ac.jp}
}

\date{March 20, 2025}

\maketitle

\begin{center}
\textbf{Abstract}
\end{center}
This paper focuses on the investigation of simple $3$-$(2^n+1,13,\lambda)$ designs admitting $\PSL(2,2^n)$ as an automorphism group. Such designs arise from the orbits of $13$-element subsets under the action of $\PSL(2,2^n)$ on the projective line $X=\GF(2^n)\cup\{\infty\}$, and any union of these orbits also forms a $3$-design. Therefore, the possible values of $\lambda$ only depend on the number of orbits with each stabilizer size. We determine the total number of orbits when $\PSL(2,2^n)$ acts on $13$-element subsets of $X$ and classify all realizable values of $\lambda$.

\begin{center}
\textbf{Main theorem}
\end{center}
There exists a simple $3$-$(2^n+1,13,\lambda)$ design with $\PSL(2,2^{n})$ as an automorphism group if and only if 
$\lambda=66i_{66}+78i_{78}+143i_{143}+429i_{429}+572i_{572}+858i_{858}+1716i_{1716}$
where $0\le i_{\lambda_{B}}\le m_{\lambda_{B}}$ for each $\lambda_{B}\in \{66,78,143,429,572,858,1716\}$, and the values $m_{\lambda_{B}}$ are explicitly determined functions of $n$ as defined in $\S 5$.

\begin{center}
\end{center}

\section{Introduction} 
Let $X$ denote a set of $v$ points, and $\mathcal{B}$ a set of $k$-subsets of $X$, called \emph{blocks}. A $t$-$(v,k,\lambda)$ design is a pair $(X,\mathcal{B})$ such that every $t$-subset of $X$ appears in exactly $\lambda$ elements of $\mathcal{B}$. A design is called \emph{simple} if no block appears more than once in $\mathcal{B}$. In this paper, we focus on simple $3$-designs obtained through the action of the projective special linear group $\PSL(2,2^n)$ on the projective line. When $q = 2^n$, $G := \PSL(2,q)$ acts sharply $3$-transitively on the projective line $X := \GF(q) \cup \{\infty\}$. This special property is fundamental to our approach: for any block $B \in \binom{X}{k}$, the orbit $G(B) := \{g(B) \mid g \in G\}$ forms a $3$-$(2^n+1,k,\lambda_B)$ design, where $\lambda_B$ is determined by the size of the stabilizer $G_B := \{g \in G \mid g(B) = B\}$. Specifically, $\lambda_B = \frac{k(k-1)(k-2)}{|G_B|}$, which establishes that $\lambda_B$ is completely determined by $|G_B|$.

Therefore, to classify all possible $3$-$(2^n+1,k,\lambda)$ designs obtainable through the action of $G$ on $X$, we need to:
\begin{enumerate}
\item identify all possible stabilizer sizes $|G_B|$ for $B \in \binom{X}{k}$,
\item determine the corresponding $\lambda_B$ values for each stabilizer size,
\item calculate the number of distinct orbits $m_{\lambda_B}$ that form $3$-$(2^n+1,k,\lambda_B)$ designs.
\end{enumerate}
The complete set of $3$-$(2^n+1,k,\lambda)$ designs can then be constructed by taking unions of these orbits.

Previous studies have systematically investigated these designs for various block sizes $k$: $k=4,5$ \cite{4}, $k=6$ \cite{5}, $k=7$ \cite{6}, $k=8$ \cite{7}, and $k=9$ \cite{8}. These cases were solved using a standard approach involving the enumeration of normalized blocks containing the set $\{0,1,\infty\}$. However, as $k$ increases, the complexity of this approach grows significantly due to the increasing number of subgroups of $\PSL(2,2^n)$ that can appear as stabilizers and the combinatorial explosion in the enumeration of blocks.

In this paper, we focus on the case $k=13$, where the possible stabilizer groups $G_B$ include $A_4$, $\mathbb{Z}_3$, and $\mathbb{Z}_2^2$. The traditional enumeration methods become particularly challenging for this case. We introduce a novel approach based on the Cauchy-Frobenius-Burnside lemma for counting orbits, which eliminates the need for explicit enumeration of blocks and their normalizations. Our main contribution is twofold:
\begin{enumerate}
\item we classify all possible values of $\lambda$ for $3$-$(2^n+1,13,\lambda)$ designs obtained through the action of $\PSL(2,2^n)$ on the projective line, by determining all possible $\lambda_B$ values and the number $m_{\lambda_B}$ of orbits for each $\lambda_B$,
\item we develop a general framework for analyzing $3$-designs admitting $\PSL(2,2^n)$ as an automorphism group, which is applicable to larger values of $k$ where the complexity of stabilizer structures increases.
\end{enumerate}

This theoretical approach provides deeper insights into the structure of these designs and establishes a more efficient method for their classification as $k$ increases.

\section{Notation and Preliminaries}

We begin by establishing notations and preliminaries used throughout this paper. Let $q=2^{n}$ be a prime power and $X=\GF(2^{n})\cup \{\infty\}$, called the \emph{projective line}. We define 
\[\frac{a}{0}=\infty, \frac{a}{\infty}=0, \infty+a=a+\infty=\infty, a\infty=\infty a=\infty,\text{ and }\quad  \dfrac{a\infty+b}{c\infty+d}=\frac{a}{c},\]
and for any $a,b,c,d \in \GF(2^{n})$ with $ad-bc\neq 0$, we define a \emph{linear function} as follows:
\[\begin{array}{rccc}
f\colon &X                     &\longrightarrow& X                     \\
        & \rotatebox{90}{$\in$}&               & \rotatebox{90}{$\in$} \\
        & x                    & \longmapsto   & \dfrac{ax+b}{cx+d}.
\end{array}\]
The set of all such mappings on $X$ forms a group, called the \emph{projective special linear group} and is denoted by $\PSL(2,2^{n})$. Throughout this paper, we denote by $G$ the group $\PSL(2,2^{n})$. By \cite{4}, the unique linear function that maps $(A,B,C)$ to $(0,1,\infty)$ is given by
\[h_{A,B,C}=\dfrac{(x+A)(B+C)}{(x+C)(B+A)}.\]

\begin{defi}\label{defi:2.1 定義集}
We use the following notation throughout this paper\,:
\begin{itemize}
\item $k=13$: the block size we are considering in this paper,
\item $q=2^n$: a prime power with $n \geq 1$,
\item $G = \PSL(2,q)$: the projective special linear group,
\item $X = \GF(2^n) \cup \{\infty\}$: the projective line,
\item for $B \in \binom{X}{13}$, $G_B := \{g \in G \mid g(B) = B\}$: the stabilizer of $B$,
\item for $B \in \binom{X}{13}$, $G(B) := \{g(B) \mid g \in G\}$: the orbit containing $B$,
\item $\lambda_B$: parameter when a single orbit $G(B)$ forms a $3$-$(2^n+1,13,\lambda_B)$ design,
\item $m_{\lambda_B}$: total number of orbits that form $3$-$(2^n+1,13,\lambda_B)$ designs,
\item $J_d$: number of conjugacy classes of elements of order $d$ in $G$,
\item $N_G(H)$: normalizer of a subgroup $H$ in $G$,
\item $C_G(g_d)$: centralizer of an element $g\in G$ of order $d$,
\item $\omega \in \GF(2^n)$: a primitive cube root of unity, which exists when $n$ is even,
\item $F(13,a,b,d)$: number of $13$-subsets that remain invariant under elements of order $d$ when $n \equiv b \pmod{a}$,
\item $B(13,a,b,d)$: a function defined as $\frac{J_d}{|C_G(g_d)|} F(k,a,b,d)$ when $n \equiv b \pmod{a}$,
\item $\chi_{13}(g)$: number of $13$-subsets of $X$ fixed by $g \in G$,
\item $\phi(d)$: Euler's totient function,
\item $N_{13}$: total number of orbits of $13$-element subsets under the action of $G$,
\item for a subgroup $Q \subset G$, $\mathcal{B}_Q := \{B \in \binom{X}{13} \mid G_B \cong Q\}$,
\item for $g \in G$, $\mathcal{B}_Q^g := \{B \in \mathcal{B}_Q \mid g(B) = B\}$,
\item $O_d := \{g \in G \mid |g| = d\}$: set of elements of order $d$ in $G$,
\item $\mathcal{H} := \{H \subset G \mid H \cong Q\}$: set of subgroups of $G$ isomorphic to $Q$.
\end{itemize}
\end{defi}

\begin{lem}[{\cite{9}}]\label{lem:2.2 デザインの自己同型群}
The group $G$ is an automorphism group of the $3$-$(2^{n}+1,13,\lambda)$ design $(X,\mathcal{B})$ if and only if $\mathcal{B}$ is a union of orbits of $k$-subsets of $X$ under $G$.
\end{lem}
Since $G$ acts sharply $3$-transitively on $X$, the orbit of any $13$-subset $B$ of $X$ forms a simple $3$-$(2^n+1,13,\lambda_B)$ design. In particular, $G$ is also an automorphism group of this design. When considering $3$-designs for which $G$ is an automorphism group, the stabilizer $G_B$ becomes very important. As is well known, the following relation holds:
\begin{equation}\label{eq:2.1 軌道固定群定理}
|G|=|G_{B}||G(B)|.
\end{equation}
\begin{lem}[{\cite{1}}]\label{lem:2.3 置換指標}
The permutation character $\chi$ for this action of $G$ on $X$ is given in Table~\ref{Table:1 置換指標}, where $\phi$ is the Euler function.
\end{lem}
\begin{table}[h]
\centering
\caption{Permutation character of $G$}
\label{Table:1 置換指標}
\renewcommand{\arraystretch}{0.7}
\setlength{\tabcolsep}{5pt}{
\begin{tabular}{|c|c|c|c|c|}
\hline
Order of $g$ & 1 & 2 & $d \mid 2^n-1, d \neq 1$ & $d \mid 2^n+1$ \\
\hline
Order of the centralizer\quad $\left(|C_{G}(g_{d})|\right)$& $6\binom{2^n+1}{3}$ & $2^n$ & $(2^n-1)$ & $(2^n+1)$ \\
\hline
Number of conjugacy classes\quad $(J_{d})$ & 1 & 1 & $\phi(d)/2$ & $\phi(d)/2$ \\
\hline
the number of fixed points by $g$\quad $(\chi(g))$ & $2^n+1$ & 1 & 2 & 0 \\
\hline
\end{tabular}
}
\end{table}

\begin{lem}[{\cite{9}}]\label{lem:2.4 G_Bとの関わり}
Let $\mathcal{D}=(X,\mathcal{B})$ be a $t$-$(v,k,\lambda)$ design with $b$ blocks and $r$ blocks containing any point. Then the following equations hold:
\[(\alpha):bk=vr, \quad (\beta):\dbinom{v}{t}\lambda=b\dbinom{k}{t}.\]
\end{lem}
An orbit $G(B)$ is a simple $3$-$(2^{n}+1,13,\lambda_B)$ design with $b=|G(B)|$ blocks. By substituting equation (\ref{eq:2.1 軌道固定群定理}), $b=|G(B)|$, and $|G|=q(q^{2}-1)$ into Lemma~\ref{lem:2.4 G_Bとの関わり}, we obtain:
\begin{equation}\label{eq:2.2 ラムダとG_B}
\lambda_B=\frac{1716}{|G_{B}|}.
\end{equation}
As $\lambda_B \in \mathbb{Z}_{>0}$, we have $|G_{B}| \mid 1716 = 2^2 \cdot 3 \cdot 11 \cdot 13$. The following possibilities can be listed as candidates for $|G_{B}|$:
\begin{equation}\label{eq:2.3 GB位数の候補}
|G_{B}| \in \{1, 2, 3, 4, 6, 11, 12, 13, 22, 26, 33, 39, 44, 52, 66, 78, 132, 143, 156, 286, 429, 572, 858, 1716\}.
\end{equation}
To determine which of these values of $|G_B|$ are actually realizable, we need to identify feasible subgroups $G_{B} \subset G$. The classification of subgroups of $G$ is given by the following lemma.
\begin{lem}[{\cite{1}}]\label{lem:2.5 dicson定理}
The subgroups of $G$ are classified as follows:\\
(a): Sylow $2$-subgroups $\mathbb{Z}_{2}^{j}$, where $j \le n$,\\
(b): cyclic groups $\mathbb{Z}_{d}$, where $d\mid 2^{n}\pm1$,\\
(c): dihedral groups of order $2d$, $D_{2d}$, where $d\mid 2^{n}\pm 1$,\\
(d): semidirect products of a Sylow $2$-subgroup $\mathbb{Z}_{2}^{j}$ and a cyclic group $\mathbb{Z}_{d}$, where $d\mid 2^{n}-1$ and $d\mid 2^{j}-1$,\\
(e): $\PSL(2,2^{h})$, where $h\mid n$ and,\\
(f): $A_{4}, A_{5}$, where $n$ is even.
\end{lem}

\begin{lem}[{\cite{4}}]\label{lem:2.6 13元集合と固定点}
Let $\chi_{13}(g)$ denote the number of $13$-subsets of $X$ fixed by $g \in G$. Then $\chi_1(g)=\chi(g)$. If $g \in G$ has order $d > 1$, then $g$ has $a=\chi(g)\leq 2$ fixed points, $b=\frac{2^n+1-a}{d}$ cycles of length $d$, and no other cycles. Thus, $\chi_{13}(g)=\binom{a}{r}\binom{b}{x}$, where $13=dx+r$ with $0\leq r < d$.
\end{lem}

\begin{prop}\label{prop:2.7 G_Bの元とnの条件}
For any $B\in \binom{X}{13}$, if $g \in G_B$, then $|g| \in \{1,2,3,11,13\}$. Furthermore, if $|g|=3$ then $n$ is even. If $|g|=11$ then $n \equiv 0 \pmod{10}$. If $|g|=13$ then $n \equiv 0 \pmod{6}$.
\end{prop}

\begin{proof}
Let $g \in G_B$. By Lemma~\ref{lem:2.6 13元集合と固定点}, we have $|g| \leq 13$. From Table~\ref{Table:1 置換指標} the only elements of even order in $G$ have order $2$. Since $|G_B| \mid 1716=2^2\cdot 3\cdot 11\cdot 13$ by equation (\ref{eq:2.2 ラムダとG_B}), we have $|g| \in \{1,2,3,11,13\}$. If $G_B$ contains an element of order $3$, then by Lemma~\ref{lem:2.6 13元集合と固定点} this element must have a fixed point in $B$. Table~\ref{Table:1 置換指標} shows that such an element must have exactly two fixed points implying $3 \mid 2^n-1$ which occurs only when $n$ is even. For an element of order $11$, it must have two fixed points in $B$ requiring $11 \mid 2^n-1$ which holds if and only if $n \equiv 0 \pmod{10}$. For elements of order $13$, both cases with zero or two fixed points are possible corresponding to $13 \mid 2^n \pm 1$ which occurs if and only if $n \equiv 0 \pmod{6}$.
\end{proof}
By Proposition~\ref{prop:2.7 G_Bの元とnの条件}, the possible orders of elements in $G_B$, their fixed points, and conditions on $n$ are summarized in Table~\ref{table:2 元の位数と固定点数とn条件}.

\begin{table}[h]
\centering
\caption{Order of elements in $G_B$, related fixed points, and conditions on $n$}
\label{table:2 元の位数と固定点数とn条件}
\renewcommand{\arraystretch}{0.7}
\setlength{\tabcolsep}{5pt} 
\begin{tabular}{|c|c|c|c|c|c|c|}
\hline
Order $|g|$ & 1 & 2 & 3 & 11 & \multicolumn{2}{c|}{13} \\
\hline
Fixed points & $2^n+1$ & 1 & 2 & 2 & 0 & 2 \\
\hline
Condition & \multirow{2}{*}{Always} & \multirow{2}{*}{Always} & $n \equiv 0$ & $n \equiv 0$ & $n \equiv 0$ & $n \equiv 6$ \\
on $n$ & & & (mod 2) & (mod 10) & (mod 12) & (mod 12) \\
\hline
\end{tabular}
\end{table}
By Lemmas~\ref{lem:2.5 dicson定理}, \ref{lem:2.6 13元集合と固定点}, Proposition~\ref{prop:2.7 G_Bの元とnの条件}, and (\ref{eq:2.3 GB位数の候補}), we can determine that 
\begin{equation}\label{eq:2.4 GB位数候補2}
|G_B| \leq 26, \quad |G_B| \notin \{ 33, 39, 44, 52, 66, 78, 132, 143, 156, 286, 429, 572, 858, 1716\}.  
\end{equation}
By (\ref{eq:2.4 GB位数候補2}) and Lemma~\ref{lem:2.5 dicson定理}, the potential values of $|G_B|$, the structure of $G_B$, and the corresponding $\lambda_B$ values are presented in Table~\ref{table:3 G_Bの候補その1}.

\begin{table}[h]
\centering
\caption{Candidates for $|G_B|$, $G_B$, and corresponding $\lambda_B$}
\label{table:3 G_Bの候補その1}
\renewcommand{\arraystretch}{0.7}
\begin{tabular}{|c|c|c|c|c|c|c|c|c|c|c|}
\hline
$|G_B|$ & 1 & 2 & 3 & 4 & 6 & 11 & 12 & 13 & 22 & 26 \\
\hline
$G_B$ & Id & $\mathbb{Z}_2$ & $\mathbb{Z}_3$ & $\mathbb{Z}_2^2$ & $D_6$ & $\mathbb{Z}_{11}$ & $A_4$ & $\mathbb{Z}_{13}$ & $D_{22}$ & $D_{26}$ \\
\hline
$\lambda_B$ & 1716 & 858 & 572 & 429 & 286 & 156 & 143 & 132 & 78 & 66 \\
\hline
\end{tabular}
\end{table}
In $\S 3$, we will demonstrate that $|G_B| \notin \{6, 11, 13\}$, further refining our understanding of the possible stabilizer groups. $\S 4$ establishes the total number of orbits of 13-element subsets. Finally, in $\S 5$, we determine $m_{\lambda_B}$ for each realizable value of $\lambda_B$. In $\S 6$, we present a conclusion that summarizes our results and states the main theorem.

\section{Order of Stabilizers of $13$-subsets}

The candidates for $|G_B|$, $G_B$ and $\lambda_B$ are determined as shown in Table~\ref{table:3 G_Bの候補その1}. In this section, we prove that $|G_B| \notin \{6, 11, 13\}$. Additionally, we determine the exact values of $m_{66}$ and $m_{78}$.

\begin{lem}[{\cite{6}}]\label{lem:3.1 軌道の一意性}
Suppose $p$ is an odd prime such that $p\mid 2^{n}-1$. If $p \leq k \leq p+2$, then there exists a unique orbit $G(B)$ such that $p \mid |G_B|$, where $B$ is a $k$-subset of $X$. Moreover, if $k$ is odd, then $2 \mid |G_B|$.
\end{lem}

\begin{prop}\label{prop:3.2 m78=1}
If $11 \mid |G_B|$, then $|G_B|=22$ and
\[
m_{78} = \begin{cases}
1 & \text{if }n \equiv 0 \pmod{10},\\
0 & \text{if }n \not\equiv 0 \pmod{10}.
\end{cases}
\]
\end{prop}

\begin{proof}
Recall that $m_{78}$ represents the number of orbits $G(B)$ of 3-$(2^{n}+1,13,\lambda_{78})$ designs. By Table~\ref{table:2 元の位数と固定点数とn条件}, the existence of an element of order $11$ requires $n \equiv 0 \pmod{10}$. If $n \not\equiv 0 \pmod{10}$, then no element of order $11$ exists in $G$. Consequently, no such element can exist in any subgroup $G_B \subset G$. This implies that $|G_B| \neq 22$ in this case, thus $m_{78} = 0$ immediately follows. Therefore, in the subsequent proof, we consider only the case where $n \equiv 0 \pmod{10}$. To prove that if $11\mid |G_{B}|$, then $|G_{B}|=22$, it suffices to show that $2\mid |G_{B}|$ by Table~\ref{table:3 G_Bの候補その1}. 

Let $11 \mid |G_B|$. Then $G_B$ contains an element of order $11$. By Table~\ref{Table:1 置換指標}, we have $11 \mid 2^n-1$. Since $k=13$, it follows from Lemma~\ref{lem:3.1 軌道の一意性} that $2 \mid |G_B|$ and $G(B)$ is the unique orbit such that $11 \mid |G_B|$, hence $m_{78}=1$.
\end{proof}

\begin{lem}[{\cite{5}}]\label{lem:3.3 位数2の元の必要十分条件}
Suppose $g(x)=\frac{ax+b}{cx+d}\in G$, then $|g|=2$ if and only if $a=d$ and $(b,c)\neq (0,0)$.
\end{lem}

\begin{remark}\label{rem:3.4}
In the following Proposition~\ref{prop:3.5 m66=1}, we note that according to Table~\ref{table:2 元の位数と固定点数とn条件}, elements of order $13$ in $G$ can occur in two distinct situations: when $n \equiv 0 \pmod{12}$, in which case $13 \mid 2^n-1$ and such elements have exactly $2$ fixed points; or when $n \equiv 6 \pmod{12}$, in which case $13 \mid 2^n+1$ and such elements have no fixed points.

For the case where $n \equiv 0 \pmod{12}$, we could apply Lemma~\ref{lem:3.1 軌道の一意性} to establish our result, similar to the approach used in Proposition~\ref{prop:3.2 m78=1}. However, this lemma is not applicable when $n \equiv 6 \pmod{12}$.

Therefore, the proof of Proposition~\ref{prop:3.5 m66=1} is presented in a manner that encompasses both cases simultaneously, regardless of whether the elements of order $13$ have fixed points or not.
\end{remark}

\begin{prop}\label{prop:3.5 m66=1}
If $13 \mid |G_B|$, then $|G_B|=26$ and
\[
m_{66} = \begin{cases}
1 & \text{if }n \equiv 0 \pmod{6},\\
0 & \text{if }n \not\equiv 0 \pmod{6}.
\end{cases}
\]
\end{prop}
\begin{proof}
Recall that $m_{66}$ represents the number of orbits $G(B)$ of 3-$(2^{n}+1,13,\lambda_{66})$ designs. By Table~\ref{table:2 元の位数と固定点数とn条件}, the existence of an element of order $13$ requires $n \equiv 0 \pmod{6}$. If $n \not\equiv 0 \pmod{6}$, then no element of order $13$ exists in $G$. Consequently, no such element can exist in any subgroup $G_B \subset G$. This implies that $|G_B| \neq 26$ in this case, thus $m_{66} = 0$ immediately follows. Therefore, in the subsequent proof, we consider only the case where $n \equiv 0 \pmod{6}$. To prove that if $13\mid |G_{B}|$, then $|G_{B}|=26$, it suffices to show that $2\mid |G_{B}|$ by Table~\ref{table:3 G_Bの候補その1}.

By Lemma~\ref{lem:2.5 dicson定理}, the Sylow $13$-subgroup is cyclic. 
Hence any subgroups of order $13$ in $G$ are conjugate to each other. Therefore it is enough to show that a particular subgroup $G_{B_{0}}$ which contains an element of order $13$ contains an involution. Let $\zeta$ be a root of $x^{6}+x^{5}+1=0$ in $\GF(2^{n})$. Note that this is a primitive polynomial over $\GF(2^{6})$, so $\zeta$ is an element of order $63$ in $\GF(2^{n})^{\times}$.

First, we show that $g(x) = \frac{x+\zeta+1}{x}$ is an element of $G$ of order $13$.
Let us define a matrix $A = \zeta^{34}\begin{psmallmatrix} 1 & \zeta+1 \\ 1 & 0 \end{psmallmatrix}=\begin{psmallmatrix} \zeta^{34} & \zeta^{29} \\ \zeta^{34} & 0 \end{psmallmatrix}$.
The determinant of $A$ is $\det(A) = \zeta^{29} \cdot \zeta^{34} = \zeta^{63} = 1$, \text{ so } $A \in \SL(2, 2^n)$. The characteristic polynomial $P_{A}(x)$ of matrix $A$ is 
\[P_A(x)=x^{2}+\operatorname{trace}(A)x+\det(A)=x^{2}+\zeta^{34}x+1.\]
Since 
\[x^{13}+1=(x+1)\prod_{i=0}^{5}(x^{2}+(\zeta^{34})^{2^{i}}+1),\]
the characteristic polynomial divides $x^{13}+1$, which confirms that the matrix has order $13$.

Now, we define a set $B_0$ as follows:
\[B_{0} = \{g^{i}(1)\}_{\{i=0,1,\ldots,12\}}=\{1,\zeta,\zeta^{62},\zeta^{39},\zeta^{26},\zeta^{29},\zeta^{32},\zeta^{19},\zeta^{59},\zeta^{57},\zeta^{58},0,\infty\} \subset X.\]
Clearly $g \in G_{B_{0}}$ and $|g| = 13$. Furthermore, we define the linear function $f(x) = \frac{x+\zeta}{x+1}$. By Lemma~\ref{lem:3.3 位数2の元の必要十分条件}, we have $|f| = 2$. We can verify that $f \in G_{B_0}$ since the relation $fg=g^{12}f$ implies that $f(g^{i}(1))=g^{13-i}(f(1))=g^{13-i}(\infty)=g^{12-i}(1)$ for $i=0,1,\ldots,12$, which demonstrates that $f$ maps $B_0$ to itself. Moreover, the relation $fgf = g^{12}$ holds, so we have $\langle f,g \rangle \cong D_{26}$.
\end{proof}

\begin{lem}[{\cite{1}}]\label{lem:3.6 二面体群は共役}
For $d$ satisfying $d \mid 2^n \pm 1$, all dihedral groups $D_{2d} \subset G$ of order $2d$ are conjugate in $G$.
\end{lem}

\begin{remark}
Regarding Lemma \ref{lem:3.6 二面体群は共役}, the condition for a dihedral group of order $2d$ to be a subgroup of $G$ is, by Lemma \ref{lem:2.5 dicson定理}, that $d \mid 2^n \pm 1$. Therefore, whenever a dihedral group $D_{2d}$ exists in $G$, all such dihedral groups are conjugate in $G$. We will use this fact to prove Proposition \ref{prop:3.8 GBは6ではない}.
\end{remark}

\begin{prop}\label{prop:3.8 GBは6ではない}
$|G_{B}| \neq 6$.
\end{prop}

\begin{proof}
From Table~\ref{table:2 元の位数と固定点数とn条件}, the condition for an element of order $3$ to be contained in $G_B$ is that $n$ is even. Therefore, if $n$ is odd, it is immediately clear that $|G_B| \neq 6$. In the following discussion, we will only consider the case where $n$ is even, that is, when elements of order $3$ can exist in $G_B$.

When $|G_B|=6$, note that from Table~\ref{table:3 G_Bの候補その1}, we have $G_B \simeq D_6$. Therefore, we assume $D_6 = \langle r,s \mid r^3=s^2=(rs)^2=\mathrm{id}\rangle \simeq G_B$. Furthermore, by Lemma~\ref{lem:2.6 13元集合と固定点}, $r$ is a product of four $3$-cycles and $s$ is a product of six transpositions, each having exactly one fixed point in $B$.

First, we show that $r$ and $s$ share the same fixed point. Let $x_s$ and $x_r$ be the fixed points of $s$ and $r$ respectively. Since $\langle r,s\rangle \simeq D_6$, we have $rs=sr^2$. By applying $rs$ and $sr^2$ to $x_r$, we obtain $r(s(x_r))=s(x_r)$. Since $x_r$ is the unique fixed point of $r$, we must have $s(x_r)=x_r$. As $x_s$ is the unique fixed point of $s$, we conclude $x_r=x_s$. Thus, $r$ and $s$ share exactly one fixed point in $B$.
Without loss of generality, we may assume this common fixed point is $\infty$. Any element of $G$ fixing $\infty$ has the matrix form $\begin{psmallmatrix} a & b \\ 0 & a^{-1} \end{psmallmatrix}$ where $a \in \GF(2^n)^\times$ and $b \in \GF(2^n)$. This is because when a linear function $f(x)=\frac{ax+b}{cx+d}$ fixes $\infty$, then $c=0$ and the determinant condition $ad-bc=ad=1$ implies $d=a^{-1}$. Elements of order $2$ fixing $\infty$ have the form $\begin{psmallmatrix} 1 & x \\ 0 & 1 \end{psmallmatrix}$ where $x \in \GF(2^n) \setminus \{0\}$, and elements of order $3$ fixing $\infty$ have the form $\begin{psmallmatrix} \alpha & y \\ 0 & \alpha^2 \end{psmallmatrix}$ where $y \in \GF(2^n)$ and $\alpha \in \{\omega,\omega^2\}$. Therefore, $D_6=G_B$ with fixed point $\infty$ must have the form $D_6=\left\langle\begin{psmallmatrix} 1 & x \\ 0 & 1 \end{psmallmatrix}, \begin{psmallmatrix} \alpha & y \\ 0 & \alpha^2 \end{psmallmatrix}\right\rangle$ for some $x,y,\alpha$.

Let $r=\begin{psmallmatrix} \alpha & y \\ 0 & \alpha^2 \end{psmallmatrix}$ and $s=\begin{psmallmatrix} 1 & x \\ 0 & 1 \end{psmallmatrix}$. By the definition of dihedral groups, we must have $srs=r^2$. However, we find that
$srs=\begin{psmallmatrix} \alpha & x+y \\ 0 & \alpha^2 \end{psmallmatrix}$ and
$r^2=\begin{psmallmatrix} \alpha^2 & y \\ 0 & \alpha \end{psmallmatrix}$. Since $\alpha \neq \alpha^2$, we have $srs \neq r^2$ for any choice of $x,y,\alpha$. Therefore, no $D_6$ with fixed point $\infty$ can exist as $G_B$.

By Lemma~\ref{lem:3.6 二面体群は共役}, all dihedral groups $D_{6}$ are conjugate in $G$. Suppose $D_6\simeq G_B$ has a fixed point $p$ in $B$. Then there exists $g \in G$ such that $g$ moves $p$ to $\infty$ by conjugation, and $gG_{B}g^{-1}=G_{g(B)}$ becomes a stabilizer of the 13-element set $g(B)$ with fixed point $\infty$. However, this contradicts our previous result that no dihedral group fixing $\infty$ can exist as $G_B$. Therefore, no dihedral group $D_6 \subset G$ can be realized as $G_B$.
\end{proof}
Summarizing the above results, we obtain the following theorem.
\begin{thm}\label{thm:main}
Suppose that $G(B)$ forms a $3$-$(2^n+1,13,\lambda_{B})$ design. Then the possible values of $\lambda_{B}$ are restricted to:
\[\lambda_{B} \in \{66,78,143,429,572,858,1716\}.\]
\end{thm}

\begin{remark}
From the discussion in $\S 3$, we have shown that $|G_{B}| \notin \{6,11,13\}$ ($\lambda_{B} \notin \{132,156,286\}$) from the candidates listed in Table~\ref{table:3 G_Bの候補その1}. Therefore, summarizing these results, we obtain Table~\ref{table:4 GB候補2} below.
\end{remark}

\begin{table}[h]
\centering
\caption{Possible $|G_B|$, $G_B$, and corresponding $\lambda_B$}
\label{table:4 GB候補2}
\renewcommand{\arraystretch}{0.7}
\begin{tabular}{|c|c|c|c|c|c|c|c|}
\hline
$|G_B|$ & 1 & 2 & 3 & 4 & 12 & 22 & 26 \\
\hline
$G_B$ & Id & $\mathbb{Z}_2$ & $\mathbb{Z}_3$ & $\mathbb{Z}_2^2$ &  $A_4$ & $D_{22}$ & $D_{26}$ \\
\hline
$\lambda_B$ & 1716 & 858 & 572 & 429 &  143  & 78 & 66 \\
\hline
\end{tabular}
\end{table}

\section{Number of orbits of $13$-subsets under $G$ }

In $\S 3$, we identified the possible values of $|G_B|$ and corresponding $\lambda_B$ for the $3$-$(2^n+1,13,\lambda_B)$ design formed by $G(B)$, as summarized in Table~\ref{table:4 GB候補2}. To prove which of these values of $\lambda_B$ are actually realizable, we need to determine the corresponding values of $m_{\lambda_B}$. In this section, we will identify the total number of orbits of $13$-element subsets under the action of $G$ to determine all values of $m_{\lambda_B}$. Note that in this $\S 4$, we use the notation $q=2^n$.

\begin{defi}\label{defi: 4.1 準備}
Let $J_d$ denote the number of conjugacy classes of elements of order $d$, and let $|C_G(g_{d})|$ denote the order of the centralizer of elements of order $d$ in $G$. The values of $J_d$ and $|C_G(g_d)|$ are uniquely determined by the value of $d$, and these values are given in Table 1. For any positive integers $a$, $b$, and $d$, we define $F(13,a,b,d)$ to be the number of $13$-subsets that remain invariant under the action of elements of order $d$ when $n \equiv b \pmod{a}$. Using this function, we define $B(13,a,b,d)$ as follows:
\[B(13,a,b,d): \text{The value of } \dfrac{J_{d}}{|C_G(g_{d})|}F(13,a,b,d) \text{ when } n\equiv b\pmod{a}.\]
The value $F(13,a,b,d)$ corresponds to $\chi_{13}(g)$ when elements of order $d$ exist. In particular, the parameters $a$ and $b$ are determined by Table~\ref{table:2 元の位数と固定点数とn条件}.
\end{defi}
The following Lemma~\ref{lem:4.2 バーンサイドの補題のもと} provides the expanded form of the binomial expressions in $\chi_{13}(g)$, making them more suitable for calculation.

\begin{lem}\label{lem:4.2 バーンサイドの補題のもと}
 For $dj\leq 13 \leq d(j+1)-1$, the following formulas hold:

(1) If $d=1$, then $F(13,1,0,1)=\dfrac{1}{13!}\prod_{i=1}^{13}(q-i+2),\quad $(2) If $d=2$, then $F(13,1,0,2)=\dfrac{1}{2^{j}j!}\prod_{i=0}^{j-1}(q-2i),$

(3) Where $d \mid 2^{n}-1:$
\[
F(13,a,b,d) = 
\begin{cases}
\dfrac{\prod_{i=0}^{j-1} (q - di - 1)}{d^j j!} & \text{if } 13 \equiv 0,2 \pmod d, \\
2\dfrac{\prod_{i=0}^{j-1} (q - di - 1)}{d^j j!} & \text{if } 13 \equiv 1 \pmod d, \\
0 & \text{if } 13 \not \equiv 0,1,2 \pmod d,
\end{cases}
\]

(4) Where $d \mid 2^{n}+1:$
\[
F(13,a,b,d) = 
\begin{cases}
\dfrac{\prod_{i=0}^{j-1} (q - di + 1)}{d^j j!} & \text{if } 13 \equiv 0 \pmod d, \\
0 & \text{if}\, 13 \not \equiv 0 \pmod d.
\end{cases}
\]

Consequently, the values of $B(13,a,b,d)$ can be expressed as follows:

(1) If $d=1$, then $B(13,1,0,1)=\dfrac{1}{13!}\prod_{i=1}^{10}(q-i-1),\quad $(2) If $d=2$, then $B(13,1,0,2)=\dfrac{1}{2^{j}j!}\prod_{i=1}^{j-1}(q-2i),$

(3) Where $d \mid 2^{n}-1:$
\[
B(13,a,b,d) = 
\begin{cases}
\dfrac{\phi(d)}{2} \dfrac{\prod_{i=1}^{j-1} (q - di - 1)}{d^j j!} & \text{if } 13 \equiv 0,2 \pmod d,  \\
\phi(d) \dfrac{\prod_{i=1}^{j-1} (q - di - 1)}{d^j j!} & \text{if } 13 \equiv 1 \pmod d,  \\
0 & \text{if } 13 \not \equiv 0,1,2 \pmod d,
\end{cases}
\]

(4) Where $d \mid 2^{n}+1:$
\[
B(13,a,b,d) = 
\begin{cases}
\dfrac{\phi(d)}{2} \dfrac{\prod_{i=1}^{j-1} (q - di + 1)}{d^j j!} & \text{if } 13 \equiv 0 \pmod d, \\
0 & \text{if }  13  \not \equiv 0 \pmod d.
\end{cases}
\]
\end{lem}

\begin{prop}\label{prop:4.3 軌道の総数N13}
Let $N_{13}$ denote the total number of orbits in a 13-element set. We define:
\[X_{1,n}=\frac{1}{6227020800}(q-2)(q-3)(q-4)(q-5)(q-6)(q-7)(q-8)(q-9)(q-10)(q-11),\]
\[X_{2,n}=\frac{1}{46080}(q-2)(q-4)(q-6)(q-8)(q-10), \quad X_{3,n}=\frac{1}{972}(q-4)(q-7)(q-10),\]
and let $X_{n}=X_{1,n}+X_{2,n}+X_{3,n}$. Then, the value of $N_{13}$ is determined as follows:
\[
N_{13} =
\begin{cases}
X_n, & \text{if } n \equiv 2,4,8,14,16,22,26,28,32,34,38,44,46,52,56,58 \pmod{60}, \\
X_n + \frac{5}{11}, & \text{if } n \equiv 10,20,40,50 \pmod{60}, \\
X_n + \frac{6}{13}, & \text{if } n \equiv 6,12,18,24,36,42,48,54 \pmod{60}, \\
X_n + \frac{131}{143}, & \text{if } n \equiv 0,30 \pmod{60}, \, n > 0, \\
X_{1,n} + X_{2,n}, & \text{if}\,\,n\,\,\text{is odd}.
\end{cases}
\]
\end{prop}

\begin{proof}
By Cauchy-Frobenius-Burnside lemma, we have:
\[N_{13}=\dfrac{1}{|G|}\left(\sum_{g\in G,|g|=1}\chi_{13}(g)+\sum_{g\in G,|g|=2}\chi_{13}(g)+\sum_{g\in G,|g|=3}\chi_{13}(g)+\sum_{g\in G,|g|=11}\chi_{13}(g)+\sum_{g\in G,|g|=13}\chi_{13}(g)\right).\]

Furthermore, since $\dfrac{1}{|G|}\displaystyle\sum_{g\in G, |g|=d}(\chi_{13}(g)) = B(13,a,b,d)$, by Lemma~\ref{lem:4.2 バーンサイドの補題のもと}, we obtain:
\begin{equation*}
\begin{cases}
B(13,1,0,1) = \frac{1}{6227020800}(q-2)(q-3)(q-4)(q-5)(q-6)(q-7)(q-8)(q-9)(q-10)(q-11) = X_{1,n}, \\
B(13,1,0,2) = \frac{1}{46080}(q-2)(q-4)(q-6)(q-8)(q-10) = X_{2,n}, \\
B(13,2,0,3) = \frac{1}{972}(q-4)(q-7)(q-10) = X_{3,n}, \\
B(13,10,0,11) = \frac{5}{11}, \\
B(13,12,0,13) = B(13,12,6,13) = \frac{6}{13}.
\end{cases}
\end{equation*}
We remark that $B(13,12,0,13) = B(13,12,6,13)$. Throughout the third case below, we will adopt the left-hand side of this equality as our notation. Depending on the value of $n$ modulo $60$, we obtain:
\begin{equation*}
N_{13} = 
\begin{cases}
B(13,1,0,1) + B(13,1,0,2) + B(13,2,0,3) = X_n, \\
\quad \text{if }n \equiv 2,4,8,14,16,22,26,28,32,34,38,44,46,52,56,58 \pmod{60}, \\[0.5em]
B(13,1,0,1) + B(13,1,0,2) + B(13,2,0,3) + B(13,10,0,11) = X_n + \dfrac{5}{11}, \\
\quad \text{if }n \equiv 10,20,40,50 \pmod{60}, \\[0.5em]
B(13,1,0,1) + B(13,1,0,2) + B(13,2,0,3) + B(13,12,0,13) = X_n + \dfrac{6}{13}, \\
\quad \text{if }n \equiv 6,12,18,24,36,42,48,54 \pmod{60}, \\[0.5em]
B(13,1,0,1) + B(13,1,0,2) + B(13,2,0,3) + B(13,10,0,11) + B(13,12,0,13) = X_{n} + \dfrac{131}{143}, \\
\quad \text{if }n \equiv 0,30 \pmod{60},\; n > 0, \\[0.5em]
B(13,1,0,1) + B(13,1,0,2) = X_{1,n} + X_{2,n}, \\
\quad \text{if}\,\,n\,\,\text{is odd}.
\end{cases}
\end{equation*}
\end{proof}

\section{Distributions of $3$-designs with block size $13$}

Recall that we have already determined that for any block $B\in \binom{X}{13}$, the orbit $G(B)$ forms a $3$-$(2^n+1, 13, \lambda_B)$ design, where $\lambda_B = \frac{1716}{|G_B|}$. By Table~\ref{table:4 GB候補2}, we obtain the possible values of $|G_B|$ and corresponding $\lambda_B$ as follows:
\[|G_B| \in \{1, 2, 3, 4, 12, 22, 26\}, \quad \lambda_B \in \{66,78,143,429,572,858,1716\}.\]

In this section, our objective is to determine the values of $m_{\lambda_B}$ for each $\lambda_B$ listed above, where $m_{\lambda_B}$ represents the number of distinct orbits that form $3$-$(2^n+1, 13, \lambda_B)$ designs. We note that the values of $m_{78}$ and $m_{66}$ have already been completely determined by Propositions~\ref{prop:3.2 m78=1} and \ref{prop:3.5 m66=1}, respectively. Therefore, we will focus on determining the values of $m_{\lambda_B}$ for the remaining parameters.

\begin{lem}[{\cite{4}}]\label{lem:5.1 連立方程式}
The following equations hold:
\[(A): \sum m_{\lambda} =N_{13}, \quad (B): \sum \lambda m_{\lambda}=\dbinom{2^{n}-2}{10}. \]
\end{lem}
By Lemma~\ref{lem:5.1 連立方程式}, we have:
\begin{align*}
    \text{(A)'} & : m_{1716} + m_{858} + m_{572} + m_{429} + m_{143} + m_{78} + m_{66} = N_{13}, \\
    \text{(B)'} & : 1716m_{1716} + 858m_{858} + 572m_{572} + 429m_{429} + 143m_{143} + 78m_{78} + 66m_{66} = \binom{2^{n}-2}{10}.
\end{align*}

\subsection{Calculation of $m_{143}$}
\begin{remark}
When $G_B \simeq A_4$, the corresponding value $\lambda_B = 143$ represents the parameter of the $3$-$(2^n+1,13,143)$ design formed by the orbit $G(B)$. We aim to determine $m_{143}$, which represents the total number of orbits forming such designs.
\end{remark}

\begin{defi}
For a given subgroup $Q\subset G$, we define:
\[\mathcal{B}_{Q}\coloneq \{B\in \binom{X}{13} \mid G_{B} \simeq Q\},\quad O_{d}\coloneq \{g\in G \mid |g|=d\},\quad \mathcal{H}\coloneq \{ H\subset G \mid H\simeq Q\}.\]
\end{defi}
From Table~\ref{Table:1 置換指標}, we have $|O_{2}|=4^{n}-1$, and for even $n$, $|O_{3}|=2^{n}(2^{n}+1)$. Note that the number of elements of order $d$ in $G$ can be calculated by $\frac{J_{d}\cdot |G|}{|N_{G}(g_{d})|}$, where $J_d$ is the number of conjugacy classes of elements of order $d$ in $G$.
By Cauchy-Frobenius-Burnside lemma, we obtain $m_{\lambda_B}=\frac{1}{|G|}\sum_{g\in G}|\mathcal{B}_{A_{4}}^{g}|$, where  $\mathcal{B}_{A_{4}}^{g} = \{B \in \mathcal{B}_{A_{4}} \mid g(B) = B\}$.

\begin{lem}\label{lem:5.4 A4の固定点はBの中で1つ}
Let $A_4 = \left<r,s\mid r^{3}=s^{2}=(sr)^{3}=\mathrm{id} \right> = G_B$. Then $r$ and $s$ have a common fixed point in $B$.
\end{lem}

\begin{proof}
Consider $A_4 = \left<r,s\mid r^3=s^2=(sr)^3=\mathrm{id}\right> = G_B$. Let $\mathbb{Z}_2^2=\{\mathrm{id}, s,rsr^2,r^2sr\} \subset A_4$ and denote $s' \coloneqq rsr^2$ and $s'' \coloneqq r^2sr$. Let $x_s$ be the fixed point of $s$. We will derive a contradiction by assuming $r(x_s) \neq x_s$. Since $s's''s=\mathrm{id}$, we have $s's''s(x_s)=x_s$. As $s's''(x_s)=x_s$, and given our assumption $r(x_s) \neq x_s$, let us denote $s''(x_s) \coloneqq a \neq x_s$. Then we must have $s'(a)=x_s$. We have $s''s'(a)=a$, and since $s's''=s$, this implies $s(a)=a$. However, this means that $s$ has two fixed points, $a$ and $x_s$. This contradicts Table 1, which shows that elements of order 2 have exactly one fixed point. Therefore, $r$ and $s$ must share the same fixed point in $A_4$, and consequently, all elements in $A_4$ have the same fixed point.
\end{proof}

\begin{lem}\label{lem:5.5 H'の中で共役}
$G$ acts transitively on 
\[\mathcal{H}' = \{H \subset G \mid H = \left<r,s\right> \simeq A_4, |r|=3, |s|=2, \text{ where $\infty$ is the common fixed-point of $r$ and $s$}\}.\]
\end{lem}

\begin{proof}
Any element of $G$ fixing $\infty$ has the matrix form $\begin{mmatrix} a & b \\ 0 & a^{-1} \end{mmatrix}$ where $a \in \GF(2^n)^\times$ and $b \in \GF(2^n)$. Because if a linear function $f(x)=\frac{ax+b}{cx+d}$ fixes $\infty$, then $c=0$ and the determinant condition $ad-bc=ad=1$ implies $d=a^{-1}$. Elements of order $2$ fixing $\infty$ have the form $\begin{mmatrix} 1 & x \\ 0 & 1 \end{mmatrix}$ where $x \in \GF(2^n) \setminus \{0\}$, and elements of order $3$ fixing $\infty$ have the form $\begin{mmatrix} \alpha & y \\ 0 & \alpha^2 \end{mmatrix}$ where $y \in \GF(2^n)$ and $\alpha \in \{\omega,\omega^2\}$. Since $A_4 = \left<r,s \mid r^3=s^2=(sr)^3=1\right>$ has elements $\{id, r, r^2, s, rs, r^2s, sr, sr^2, rsr, rsr^2, r^2sr, r^2sr^2\}$, $\left<\begin{mmatrix} \alpha & y \\ 0 & \alpha^2 \end{mmatrix}, \begin{mmatrix} 1 & x \\ 0 & 1 \end{mmatrix}\right>$ is isomorphic to $A_4$. A subgroup of $G$ isomorphic to $A_4$ fixes $\infty$ in $B$ if and only if it is of the form 
\begin{equation}\label{eq:5.1 無限を固定するA4と同型な行列群の形}
\left\langle\begin{mmatrix} \alpha & y \\ 0 & \alpha^2 \end{mmatrix}, \begin{mmatrix} 1 & x \\ 0 & 1 \end{mmatrix}\right\rangle, x\in \GF(2^{n})^{\times}, y\in \GF(2^{n}), \alpha\in \{\omega,\omega^{2}\}.  
\end{equation}
For any $\gamma \in \GF(2^n)^{\times}$, let $c := (\gamma x^{-1})^{2^{n-1}}$. 
Note that $c^2 = (\gamma x^{-1})^{2^n} = \gamma x^{-1}.$
For any $\delta \in \GF(2^n)$, let $b := \delta + c^2y$.
Define $g = \begin{mmatrix} c & bc^{-1} \\ 0 & c^{-1} \end{mmatrix} \in G$. By direct calculation, we have:
\[g\begin{mmatrix} 1 & x \\ 0 & 1 \end{mmatrix}g^{-1} = \begin{mmatrix} 1 & \gamma \\ 0 & 1 \end{mmatrix},\quad 
g\begin{mmatrix} \alpha & y \\ 0 & \alpha^2 \end{mmatrix}g^{-1} = \begin{mmatrix} \alpha & c^2y+b \\ 0 & \alpha^2 \end{mmatrix} = \begin{mmatrix} \alpha & \delta \\ 0 & \alpha^2 \end{mmatrix}.\] Therefore, for any $\gamma \in \GF(2^n)^{\times}$ and any $\delta \in \GF(2^n)$, we can find a $g \in G$ such that:
\[g\left\langle\begin{mmatrix} \alpha & y \\ 0 & \alpha^2 \end{mmatrix}, \begin{mmatrix} 1 & x \\ 0 & 1 \end{mmatrix}\right\rangle g^{-1} = \left\langle\begin{mmatrix} \alpha & \delta \\ 0 & \alpha^2 \end{mmatrix}, \begin{mmatrix} 1 & \gamma \\ 0 & 1 \end{mmatrix}\right\rangle.\]
\end{proof}

\begin{lem}\label{lem:5.6 すべてのA4はGの中で共役}
$G$ acts transitively on
\[\mathcal{H} = \{H \subset G \mid H = \left<r,s\right> \simeq A_4\}.\]
\end{lem}
\begin{proof}
Let $H = \left<\begin{psmallmatrix} 1 & x \\ 0 & 1 \end{psmallmatrix}, \begin{psmallmatrix} \alpha & y \\ 0 & \alpha^2 \end{psmallmatrix}\right>$ for some $x\in \GF(2^{n})^{\times}, y \in \GF(2^n), \alpha \in \{\omega,\omega^2\}$. According to (\ref{eq:5.1 無限を固定するA4と同型な行列群の形}), this is the form of any subgroup of $G$ isomorphic to $A_4$ fixing $\infty$ in $B$. Let $H_1$ be any subgroup of $G$ isomorphic to $A_4$ with fixed point $p_1 \in B$.
By Lemma~\ref{lem:5.4 A4の固定点はBの中で1つ}, any $A_4$ subgroup of $G$ that appears as $G_B$ has generators sharing a common fixed point. Since $G$ acts $3$-transitively on $X$, there exists $g \in G$ such that $g(p_1) = \infty$. Then $gH_1g^{-1}$ is a subgroup of $G$ isomorphic to $A_4$ fixing $\infty$. By Lemma~\ref{lem:5.5 H'の中で共役}, $gH_1g^{-1}$ is conjugate to $H$ in $G$. Therefore, any subgroup of $G$ isomorphic to $A_4$ is conjugate to $H$, implying that $G$ acts transitively on $\mathcal{H}$.
\end{proof}

\begin{lem}\label{lem:5.7 Hの正規化群}
Let $\mathcal{H}=\{H\subset G\mid H=\langle r,s\rangle\simeq A_4\}$. Then $H$ is self-normalizing.
\end{lem}

\begin{proof}
By (\ref{eq:5.1 無限を固定するA4と同型な行列群の形}) of Lemma~\ref{lem:5.5 H'の中で共役}, when $A_4=G_B$ and $\infty$ is a fixed point in $B$, then $A_4$ can be written as $\langle\begin{psmallmatrix} 1 & x \\ 0 & 1 \end{psmallmatrix}, \begin{psmallmatrix} \alpha & y \\ 0 & \alpha^2 \end{psmallmatrix}\rangle $ for some $ x\in \GF(2^n)^\times, y\in \GF(2^n), \alpha\in \{\omega,\omega^2\}$. Furthermore, by Lemma~\ref{lem:5.6 すべてのA4はGの中で共役}, all subgroups isomorphic to $A_4$ are conjugate in $G$. Therefore, it suffices to prove $N_G(H)=H$ for $H=\langle\begin{psmallmatrix} 1 & 1 \\ 0 & 1 \end{psmallmatrix}, \begin{psmallmatrix} \omega^2 & 0 \\ 0 & \omega \end{psmallmatrix}\rangle$.
Since $g\in N_G(H)$ normalizes $H$, we can assume $g=\begin{psmallmatrix} g_1 & g_2 \\ 0 & g_1^{-1} \end{psmallmatrix}$ where $g_1\in \GF(2^n)^\times$ and $g_2\in \GF(2^n)$. By conjugation:
\[g\begin{psmallmatrix} \omega^2 & 0 \\ 0 & \omega \end{psmallmatrix}g^{-1} = \begin{psmallmatrix} \omega^2 & g_1g_2 \\ 0 & \omega \end{psmallmatrix} \in H.\]
This implies $g_1g_2 \in \{0,1,\omega,\omega^2\}$. Therefore, $g_2=jg_1^{-1}$ for some $j\in \{0,1,\omega,\omega^2\}$, giving $g = \begin{psmallmatrix} g_1 & jg_1^{-1} \\ 0 & g_1^{-1} \end{psmallmatrix}.$
For any $x\in \{0,1,\omega,\omega^2\}$, we have $g\begin{psmallmatrix} 1 & x \\ 0 & 1 \end{psmallmatrix}g^{-1} = \begin{psmallmatrix} 1 & xg_1^2 \\ 0 & 1 \end{psmallmatrix} \in H.$
This implies $xg_1^2 \in \{0,1,\omega,\omega^2\}$, and therefore $g_1\in \{1,\omega,\omega^2\}$. Hence $N_G(H)=\{\begin{psmallmatrix} g_1 & jg_1^{-1} \\ 0 & g_1^{-1} \end{psmallmatrix}\mid j\in \{0,1,\omega,\omega^2\}, g_1\in \{1,\omega,\omega^2\}\}=H$.
\end{proof}

\begin{remark}\label{rem:5.8 バーンサイドの補題適用の仕方}
To determine $m_{143}$, we apply the Cauchy-Frobenius-Burnside lemma, which gives:
\[m_{143} = \frac{1}{|G|}\sum_{g\in G}|\mathcal{B}_{A_4}^{g}|.\]
Since elements of order $2$ and $3$ are all conjugate within their respective orders (as shown in Table~\ref{Table:1 置換指標}), and considering that $|O_2| = 4^n-1$ and $|O_3| = 2^n(2^n+1)$ for even $n$, we can simplify this to:
\[m_{143} = \frac{1}{|G|}\left(|\mathcal{B}_{A_4}| + (4^n - 1)|\mathcal{B}_{A_4}^{g_1}| + 2^n(2^n + 1)|\mathcal{B}_{A_4}^{g_2}|\right),\]
where $g_1=\begin{psmallmatrix} 1 & 1 \\ 0 & 1 \end{psmallmatrix}$and $g_2=\begin{psmallmatrix} \omega^2 & 0 \\ 0 & \omega \end{psmallmatrix}$.
\end{remark}

\begin{lem}\label{lem:5.9 BA4}
Let $\mathcal{B}_{A_{4}}\coloneq \{B\in \binom{X}{13} \mid G_{B} \simeq A_{4}\}$. Then
\[|\mathcal{B}_{A_4}| = \begin{cases} 
\dfrac{2^n(4^n-1)(2^n-4)}{144} & \text{if}\,\,n\,\,\text{is even},  \\
0 & \text{if}\,\,n\,\,\text{is odd}.
\end{cases}\]
\end{lem}

\begin{proof}
Let $\mathcal{H}=\{H\subset G\mid H\simeq A_4\}$, and $\mathcal{B}_{A_{4}}\coloneq \{B\in \binom{X}{13} \mid G_{B} \simeq A_{4}\}$. By Table~\ref{table:2 元の位数と固定点数とn条件}, when $n$ is odd, there are no elements of order $3$ in $G$, and therefore $\mathcal{B}_{A_4} = 0$. For even $n$, we will determine the value of $\mathcal{B}_{A_4}$.

Under the conjugation action of $G$ on $\mathcal{H}$, the stabilizer of $H\in \mathcal{H}$ coincides with its normalizer $N_G(H)$. By Lemma~\ref{lem:5.7 Hの正規化群}, we have $|N_G(H)|=|H|=12$. Therefore, by the orbit-stabilizer theorem:
\[|\mathcal{H}| = \dfrac{|G|}{|N_G(H)|} = \dfrac{2^n(4^n-1)}{12}.\]
By Lemma~\ref{lem:5.6 すべてのA4はGの中で共役}, all subgroups isomorphic to $A_4$ are conjugate in $G$, we can assume $H=\left\langle\begin{psmallmatrix} 1 & 1 \\ 0 & 1 \end{psmallmatrix}, \begin{psmallmatrix} \omega^2 & 0 \\ 0 & \omega \end{psmallmatrix}\right\rangle$. For $a\in \GF(2^{n})$, $B\setminus\{\infty\}=H(a)$ if and only if $a\notin\{0,1,\omega,\omega^2\}$, giving $2^n-4$ possible choices for $a$. Since $H$ acts transitively on $B\setminus\{\infty\}$, the number of possible choices for $B$ given $H$ is $\frac{2^n-4}{12}$. For each $H\in \mathcal{H}$, we have $\frac{2^n-4}{12}$ possible choices for $B$. Therefore:
\[|\mathcal{B}_{A_4}| = |\mathcal{H}| \cdot \dfrac{2^n-4}{12} = \dfrac{2^n(4^n-1)}{12}\cdot \dfrac{2^n-4}{12} = \dfrac{2^n(4^n-1)(2^n-4)}{144}.\]
\end{proof}

\begin{lem}\label{lem:5.10 BA4g1}
Let $g_1 = \begin{psmallmatrix} 1 & 1 \\ 0 & 1 \end{psmallmatrix}$,  $\mathcal{B}_{A_4}^{g_1} = \{B\in \binom{X}{13} \mid G_B \simeq A_4, g_1(B)=B\}$. Then
\[|\mathcal{B}_{A_4}^{g_1}| = \begin{cases} 
\dfrac{2^n(2^n-4)}{48} & \text{if}\,\,n\,\,\text{is even},  \\
0 & \text{if}\,\,n\,\,\text{is odd}.
\end{cases}\]

\end{lem}

\begin{proof}
We write $\mathcal{B}_{A_4}^{g_1} = \{B\in \binom{X}{13} \mid G_B \simeq A_4, g_1(B)=B\}$. By Table~\ref{table:2 元の位数と固定点数とn条件}, when $n$ is odd, there are no elements of order $3$ in $G$, and therefore $\mathcal{B}_{A_4}^{g_{1}} = 0$. For even $n$, we will determine the value of $\mathcal{B}_{A_4}^{g_{1}}$. By Lemma~\ref{lem:5.6 すべてのA4はGの中で共役}, all subgroups isomorphic to $A_4$ are conjugate in $G$. Moreover, Lemma~\ref{lem:5.5 H'の中で共役} shows that when $\infty$ is a fixed point in $B$, we can assume $H=\left\langle g_1,\begin{psmallmatrix} \omega^2 & 0 \\ 0 & \omega \end{psmallmatrix}\right\rangle$. We write $\mathcal{H}_{g_{1}}=\{H\subset G\mid H\simeq A_4, g_{1}\in H\}$.

We first determine the centralizer $C\coloneq C_G(g_1)$. Any element $h\in C$ has the form $h=\begin{psmallmatrix} h_1 & h_2 \\ 0 & h_1^{-1} \end{psmallmatrix}$ with $h_1\in \GF(2^n)^\times$ and $h_2\in \GF(2^n)$. The conjugation equation $hg_1h^{-1}=\begin{psmallmatrix} 1 & h_1^2 \\ 0 & 1 \end{psmallmatrix}=g_1$
yields $h_1=1$. Thus we have
$C=\left\{\begin{psmallmatrix} 1 & h_2 \\ 0 & 1 \end{psmallmatrix}\mid h_2\in \GF(2^n)\right\}.$

We next determine $N_C(H)$. For $\begin{psmallmatrix} 1 & h_2 \\ 0 & 1 \end{psmallmatrix}\in C$, conjugation gives $\begin{psmallmatrix} 1 & h_2 \\ 0 & 1 \end{psmallmatrix}\begin{psmallmatrix} \omega^2 & 0 \\ 0 & \omega \end{psmallmatrix}\begin{psmallmatrix} 1 & h_2 \\ 0 & 1 \end{psmallmatrix}=\begin{psmallmatrix} \omega^2 & h_2 \\ 0 & \omega \end{psmallmatrix}\in H.$
This equation implies $h_2\in \{0,1,\omega,\omega^2\}$. Therefore we obtain
$N_C(H)=\left\{\begin{psmallmatrix} 1 & h_2 \\ 0 & 1 \end{psmallmatrix}\mid h_2\in \{0,1,\omega,\omega^2\}\right\}\simeq \mathbb{Z}_2^2.$ Under conjugation by elements of $C$, the stabilizer of $H\in \mathcal{H}_{g_{1}}$ is $N_C(H)$. The orbit-stabilizer theorem gives
\[|\mathcal{H}_{g_{1}}|=\frac{|C|}{|N_C(H)|}=2^{n-2}.\]
Lemma~\ref{lem:5.9 BA4} shows that each $H\in \mathcal{H}_{g_{1}}$ allows $\frac{2^n-4}{12}$ possible choices for $B$. We thus conclude
\[|\mathcal{B}_{A_4}^{g_1}|=|\mathcal{H}_{g_{1}}| \cdot \frac{2^n-4}{12} =2^{n-2}\cdot \frac{2^n-4}{12}=\frac{2^n(2^n-4)}{48}.\]
\end{proof}

\begin{lem}\label{lem:5.11 BA4g2}
Let $g_2 = \begin{psmallmatrix} \omega^2 & 0 \\ 0 & \omega \end{psmallmatrix}$, $\mathcal{B}_{A_4}^{g_2} = \{B\in \binom{X}{13} \mid G_B \simeq A_4, g_2(B)=B\}.$ Then 
\[|\mathcal{B}_{A_4}^{g_2}| = \begin{cases} 
\dfrac{(2^n-1)(2^n-4)}{18} & \text{if}\,\,n\,\,\text{is even},  \\
0 & \text{if}\,\,n\,\,\text{is odd}.
\end{cases}\]

\end{lem}

\begin{proof}
We write $\mathcal{B}_{A_4}^{g_2} = \{B \in \binom{X}{13} \mid G_B \simeq A_4, g_2(B)=B\}$ and let $\mathcal{H}_{g_{2}}=\{H\subset G\mid H\simeq A_4, g_2\in H\}$. By Table~\ref{table:2 元の位数と固定点数とn条件}, when $n$ is odd, there are no elements of order $3$ in $G$, and therefore $\mathcal{B}_{A_4}^{g_{2}} = 0$. For even $n$, we will determine the value of $\mathcal{B}_{A_4}^{g_{2}}$.

By Lemma~\ref{lem:5.6 すべてのA4はGの中で共役}, all subgroups isomorphic to $A_4$ are conjugate in $G$. We observe that $g_2$ has two fixed points $0$ and $\infty$. When $\infty$ is a fixed point by $H\in \mathcal{H}_{g_{2}}$, Lemma~\ref{lem:5.5 H'の中で共役} shows we can write $H=\left\langle \begin{psmallmatrix} 1 & x \\ 0 & 1 \end{psmallmatrix}, g_2 \right\rangle \text{for}\, x\in \GF(2^n)^\times$. Similarly, when $0$ is a fixed point by $H\in \mathcal{H}_{g_{2}}$, we can write $H=\left\langle \begin{psmallmatrix} 1 & 0 \\ x & 1 \end{psmallmatrix}, g_2 \right\rangle \text{for}\, x\in \GF(2^n)^\times$. Let $\mathcal{H}_{g_2,\infty}=\{H\subset G\mid H\simeq A_4, g_2\in H, \text{ fixed-point of }H \text{ is }\infty\}$ and $\mathcal{H}_{g_2,0}=\{H\subset G\mid H\simeq A_4, g_2\in H, \text{ fixed-point of }H \text{ is }0\}$. Then $\mathcal{H}_{g_{2}}=\mathcal{H}_{g_2,\infty}\cup \mathcal{H}_{g_2,0}$. Since $\mathcal{H}_{g_2,\infty}$ and $\mathcal{H}_{g_2,0}$ are conjugate in $G$, it suffices to analyze $\mathcal{H}_{g_2,\infty}$.

We first determine $C\coloneq C_G(g_2)$. By Lemma~\ref{lem:5.5 H'の中で共役}, since all subgroups in $\mathcal{H}_{g_2,\infty}$ are conjugate, we may assume without loss of generality that 
$H=\left\langle\begin{psmallmatrix} 1 & 1 \\ 0 & 1 \end{psmallmatrix},g_{2}\right\rangle$.\\
Let $h\in C$ have the form $h=\begin{psmallmatrix} h_1 & h_2 \\ h_3 & h_4 \end{psmallmatrix}$. The conjugation equation
$hg_2h^{-1}=\begin{psmallmatrix} \omega^2h_1h_4+\omega h_2h_3 & h_1h_2 \\ h_3h_4 & \omega h_{1}h_{4}+\omega^{2}h_{2}h_{3} \end{psmallmatrix}=g_2$
yields $h_1h_2=h_3h_4=0$. We consider two cases: (A) $h_1=h_4=0$ and (B) $h_2=h_3=0$. Case (A) leads to contradictory equations $h_2h_3=\omega$ and $h_2h_3=\omega^2$. For case (B), we obtain $h_1h_4=1$, thus
$C=\left\{\begin{psmallmatrix} h_1 & 0 \\ 0 & h_1^{-1} \end{psmallmatrix}\mid h_1\in \GF(2^n)^\times\right\}.$ 

Next, we determine $N_C(H)$. For $h=\begin{psmallmatrix} h_1 & 0 \\ 0 & h_1^{-1} \end{psmallmatrix}\in C$, conjugation gives
$h\begin{psmallmatrix} 1 & 1 \\ 0 & 1 \end{psmallmatrix}h^{-1}=\begin{psmallmatrix} 1 & h_1^2 \\ 0 & 1 \end{psmallmatrix}\in H.$
This implies $h_1^2\in \{1,\omega,\omega^2\}$, equivalent to $h_1\in \{1,\omega,\omega^2\}$. Therefore,
$N_C(H)=\{1, g_2, g_2^2\}\simeq \mathbb{Z}_3.$ Under conjugation by elements of $C$, the stabilizer of $H\in \mathcal{H}$ is $N_C(H)$. By the orbit-stabilizer theorem, we obtain
\[|\mathcal{H}_{g_2,\infty}|=\frac{|C|}{|N_C(H)|}=\frac{2^n-1}{3}.\]
A symmetric argument shows $|\mathcal{H}_{g_2,0}|=\frac{2^n-1}{3}$, where $\mathcal{H}_{g_2,\infty}$ and $\mathcal{H}_{g_2,0}$ correspond to the cases where $\infty$ and $0$ are fixed points in $B$, respectively. By Lemma~\ref{lem:5.9 BA4}, each $H\in \mathcal{H}$ allows $\frac{2^n-4}{12}$ possible choices for $B$. Therefore, we obtain
\[|\mathcal{B}_{A_4}^{g_2}|=2\cdot |\mathcal{H}_{g_2,\infty}|\cdot \frac{2^{n}-4}{12}=2\cdot \frac{2^n-1}{3}\cdot \frac{2^n-4}{12}=\frac{(2^n-1)(2^n-4)}{18}.\]
\end{proof}

\begin{prop}\label{prop:5.12 m143}
\[
m_{143} = \begin{cases}
\dfrac{2^{n}-4}{12} & \text{if}\,\,n \,\,\text{is even}, \\
0 & \text{if}\,\,n\,\,\text{is odd}.
\end{cases}
\]
\end{prop}

\begin{proof}
Note that $|O_{2}| = 4^{n}-1$, and for even $n$, $|O_{3}| = 2^{n}(2^{n}+1)$. By Table~\ref{table:2 元の位数と固定点数とn条件}, when $n$ is odd, there are no elements of order $3$ in $G$, and therefore $m_{143} = 0$.

For even $n$, by the Cauchy-Frobenius-Burnside lemma:
\[m_{143} = \dfrac{1}{|G|}\left(|\mathcal{B}_{A_4}| + (4^n-1)|\mathcal{B}_{A_4}^{g_1}| + 2^n(2^n+1)|\mathcal{B}_{A_4}^{g_2}|\right).\]
By Lemmas~\ref{lem:5.9 BA4}, \ref{lem:5.10 BA4g1}, and \ref{lem:5.11 BA4g2}, we have:
\[|\mathcal{B}_{A_4}| = \dfrac{2^n(4^n-1)(2^n-4)}{144}, \quad |\mathcal{B}_{A_4}^{g_1}| = \dfrac{2^n(2^n-4)}{48}, \quad |\mathcal{B}_{A_4}^{g_2}| = \dfrac{(2^n-1)(2^n-4)}{18}.\]
Substituting these values and $|G| = 2^n(4^n-1)$ into the formula yields the desired result.
\end{proof}

\subsection{Calculation of $m_{429}$}

\begin{remark}
When $G_B \simeq \mathbb{Z}_{2}^{2}$, the corresponding value $\lambda_B = 429$ represents the parameter of the $3$-$(2^n+1,13,429)$ design formed by the orbit $G(B)$. We aim to determine $m_{429}$, which represents the total number of orbits forming such designs.
\end{remark}
By Lemma~\ref{lem:5.4 A4の固定点はBの中で1つ}, following Lemma~\ref{lem:5.14 Z22はB内に同一の固定点を持つ} holds.

\begin{lem}\label{lem:5.14 Z22はB内に同一の固定点を持つ}
Let $\mathbb{Z}_2^2=\langle\sigma_1,\sigma_2\rangle\simeq G_B$. Then $\mathbb{Z}_2^2$ has a fixed point in $B$, and in particular, any fixed point of $\sigma_1$ is also a fixed point of $\sigma_2$.
\end{lem}

\begin{lem}\label{lem:5.15 無限を固定するZ22は1101が入る形に共役である}
Any subgroup $\mathbb{Z}_{2}^2 \subset G$ having $\infty$ as a fixed point  is conjugate in $G$ to $\langle\begin{psmallmatrix} 1 & 1 \\ 0 & 1 \end{psmallmatrix},\begin{psmallmatrix} 1 & z \\ 0 & 1 \end{psmallmatrix}\rangle$ for some $z\in \GF(2^n)\setminus \{0,1\}$.
\end{lem}

\begin{proof}
By Lemma~\ref{lem:5.5 H'の中で共役}, elements of order $2$ fixing $\infty$ must be of the form $\begin{psmallmatrix} 1 & x \\ 0 & 1 \end{psmallmatrix}$ where $x\in \GF(2^n)^{\times}$. By Lemma~\ref{lem:5.14 Z22はB内に同一の固定点を持つ}, elements of order 2 in $\mathbb{Z}_{2}^{2}$ share their fixed points. Thus any $\mathbb{Z}_{2}^{2}$ fixing $\infty$ can be written as $H \coloneqq \langle\begin{psmallmatrix} 1 & x \\ 0 & 1 \end{psmallmatrix},\begin{psmallmatrix} 1 & y \\ 0 & 1 \end{psmallmatrix} \rangle$ for $x,y\in \GF(2^n)\setminus \{0,1\}$.\\
Consider conjugation of $H$. For $g\in N_G(H)$, we may assume $g=\begin{psmallmatrix} g_1 & g_2 \\ 0 & g_1^{-1} \end{psmallmatrix}$. Then $g \begin{psmallmatrix} 1 & x \\ 0 & 1 \end{psmallmatrix} g^{-1}=\begin{psmallmatrix} 1 & g_1^2x \\ 0 & 1 \end{psmallmatrix}$. Taking $g_1=x^{-1/2}=x^{-2^{n-1}}$, we obtain $g \begin{psmallmatrix} 1 & x \\ 0 & 1 \end{psmallmatrix} g^{-1}=\begin{psmallmatrix} 1 & 1 \\ 0 & 1 \end{psmallmatrix}$.

Moreover, $g \begin{psmallmatrix} 1 & y \\ 0 & 1 \end{psmallmatrix} g^{-1}=\begin{psmallmatrix} 1 & x^{-1}y \\ 0 & 1 \end{psmallmatrix}$. Setting $z\coloneqq x^{-1}y$, we see that $H$ is conjugate to $\langle\begin{psmallmatrix} 1 & 1 \\ 0 & 1 \end{psmallmatrix},\begin{psmallmatrix} 1 & z \\ 0 & 1 \end{psmallmatrix}\rangle$ where $z\in \GF(2^n)\setminus \{0,1\}$.
\end{proof}

\begin{lem}\label{lem:5.16 Z22完全代表系}
The number of conjugacy classes of subgroups isomorphic to $\mathbb{Z}_{2}^2$ in $G$ fixing $\infty$ is given by:
\[ \begin{cases} 
\dfrac{2^{n}+2}{6} & \text{if}\,\,n\,\,\text{is even}, \\
\dfrac{2^{n}-2}{6} & \text{if}\,\,n\,\,\text{is odd}.
\end{cases} \]
\end{lem}

\begin{proof}
Let $\mathcal{H}_{\infty}=\{H\subset G \mid H\simeq \mathbb{Z}_{2}^{2}, H \text{ stabilizes } \infty \}$ and $\mathcal{H}_{\infty}'=\{H\in \mathcal{H}_{\infty} \mid \begin{psmallmatrix} 1 & 1 \\ 0 & 1 \end{psmallmatrix}\in H \}$. Since any subgroup of $G$ isomorphic to $\mathbb{Z}_{2}^{2}$ fixing $\infty$ is conjugate to $H=\langle\begin{psmallmatrix} 1 & 1 \\ 0 & 1 \end{psmallmatrix},\begin{psmallmatrix} 1 & a \\ 0 & 1 \end{psmallmatrix}\rangle=\{\begin{psmallmatrix} 1 & x \\ 0 & 1 \end{psmallmatrix}\mid x\in \{0,1,a,a+1\}\}$ for some $a\in \GF(2^n)\setminus\{0,1\}$ by Lemma~\ref{lem:5.15 無限を固定するZ22は1101が入る形に共役である}, the number of conjugacy classes of $\mathcal{H}_{\infty}$ is equal to that of $\mathcal{H}_{\infty}'$.
Therefore, it suﬃces to determine the number of distinct conjugacy classes of the set $\mathcal{H}_{\infty}'$ of groups.

Consider conjugation $g^{-1}H g=\langle\begin{psmallmatrix} 1 & 1 \\ 0 & 1 \end{psmallmatrix},\begin{psmallmatrix} 1 & a' \\ 0 & 1 \end{psmallmatrix}\rangle$ where $a'\in \GF(2^n)\setminus\{0,1\}$. Since both subgroups consist of upper triangular matrices, we can assume $g=\begin{psmallmatrix} g_1 & g_2 \\ 0 & g_1^{-1} \end{psmallmatrix}$. Under conjugation, we have $g^{-1}\begin{psmallmatrix} 1 & x \\ 0 & 1 \end{psmallmatrix}g=\begin{psmallmatrix} 1 & g_1^{-2}x \\ 0 & 1 \end{psmallmatrix}$ for $x\in \{0,1,a,a+1\}$. Thus $g^{-1}Hg=\left\{\begin{psmallmatrix} 1 & 0 \\ 0 & 1 \end{psmallmatrix},\begin{psmallmatrix} 1 & g_1^{-2} \\ 0 & 1 \end{psmallmatrix},\begin{psmallmatrix} 1 & g_1^{-2}a \\ 0 & 1 \end{psmallmatrix},\begin{psmallmatrix} 1 & g_1^{-2}(a+1) \\ 0 & 1 \end{psmallmatrix}\right\}$. Since this subgroup must contain $\begin{psmallmatrix} 1 & 1 \\ 0 & 1 \end{psmallmatrix}$, we have $g_1^{-2}\in \{1,a^{-1},(a+1)^{-1}\}$.

This leads to a conjugacy class of $\mathcal{H}_{\infty}'$ represented by $\{H_{1,a}, H_{2,a}, H_{3,a}\}$ where
\begin{align*}
H_{1,a}&=\left\{\begin{psmallmatrix} 1 & 0 \\ 0 & 1 \end{psmallmatrix},\begin{psmallmatrix} 1 & 1 \\ 0 & 1 \end{psmallmatrix},\begin{psmallmatrix} 1 & a \\ 0 & 1 \end{psmallmatrix},\begin{psmallmatrix} 1 & a+1 \\ 0 & 1 \end{psmallmatrix}\right\}, \\
H_{2,a}&=\left\{\begin{psmallmatrix} 1 & 0 \\ 0 & 1 \end{psmallmatrix},\begin{psmallmatrix} 1 & 1 \\ 0 & 1 \end{psmallmatrix},\begin{psmallmatrix} 1 & a^{-1} \\ 0 & 1 \end{psmallmatrix},\begin{psmallmatrix} 1 & a^{-1}+1 \\ 0 & 1 \end{psmallmatrix}\right\}, \\
H_{3,a}&=\left\{\begin{psmallmatrix} 1 & 0 \\ 0 & 1 \end{psmallmatrix},\begin{psmallmatrix} 1 & 1 \\ 0 & 1 \end{psmallmatrix},\begin{psmallmatrix} 1 & (a+1)^{-1} \\ 0 & 1 \end{psmallmatrix},\begin{psmallmatrix} 1 & (a+1)^{-1}+1 \\ 0 & 1 \end{psmallmatrix}\right\}.
\end{align*}
The size of this conjugacy class is $1$ if and only if $a\in \{\omega,\omega^2\}$. Any subgroup of $G$ isomorphic to $\mathbb{Z}_{2}^2$ fixing $\infty$ is conjugate to either $\langle\begin{psmallmatrix} 1 & 1 \\ 0 & 1 \end{psmallmatrix}, \begin{psmallmatrix} 1 & \omega \\ 0 & 1 \end{psmallmatrix}\rangle$ or $H_{1,a}$ where $a\in \GF(2^n)\setminus\{0,1,\omega,\omega^2\}$. 

For even $n$, observe that we can choose $a$ from $\GF(2^n)\setminus\{0,1,\omega,\omega^2\}$, giving $2^n-4$ possibilities. Since replacing $a$ by $a+1$ yields the same group, we obtain $\frac{2^{n}-4}{2}$ distinct groups, three of them make a conjugacy class except when $a\in \{\omega,\omega^2\}$. Therefore, the total number of conjugacy classes is $\frac{1}{3}\cdot \frac{2^n-4}{2}+1=\frac{2^{n}+2}{6}$. For odd $n$, since no primitive cube roots exist in $\GF(2^n)$, we only exclude $a\in \{0,1\}$, resulting in $\frac{1}{3}\cdot \frac{2^n-2}{2}=\frac{2^{n}-2}{6}$ conjugacy classes.
\end{proof}

\begin{lem}\label{lem:5.17 BZ22}
Let $\mathcal{B}_{\mathbb{Z}_{2}^{2}}=\{B\in \binom{X}{13} \mid G_B \simeq \mathbb{Z}_{2}^{2}\}$. Then
\[|\mathcal{B}_{\mathbb{Z}_{2}^{2}}| = \begin{cases} 
\dfrac{2^{2n}(4^n-1)(2^n-4)(2^n-10)}{2304} & \text{if}\,\,n\,\,\text{is even},  \\
\dfrac{2^n(4^n-1)(2^n-2)(2^n-4)(2^n-8)}{2304} & \text{if}\,\,n\,\,\text{is odd}.
\end{cases}\]
\end{lem}

\begin{proof}
Let $\mathcal{B}_{\mathbb{Z}_{2}^{2}}=\{B\in \binom{X}{13} \mid G_B \simeq \mathbb{Z}_{2}^{2}\}$.
First, we consider the case where $n$ is even. By Lemma~\ref{lem:5.16 Z22完全代表系}, we have two types of cases to consider: the complete representatives $\{H_{1,a}, H_{2,a}, H_{3,a}\}$ and the case $\left\langle\begin{psmallmatrix} 1 & 1 \\ 0 & 1 \end{psmallmatrix},\begin{psmallmatrix} 1 & \omega \\ 0 & 1 \end{psmallmatrix}\right\rangle$.

(A): We first consider the case $a \notin \{0,1,\omega,\omega^2\}$ with $H \coloneq H_{1,a} = \left\langle\begin{psmallmatrix} 1 & 1 \\ 0 & 1 \end{psmallmatrix},\begin{psmallmatrix} 1 & a \\ 0 & 1 \end{psmallmatrix}\right\rangle$. It is clear that these $H_{1,a}, H_{2,a}, H_{3,a}$ fix $\infty$ in $B$. By Lemma~\ref{lem:5.5 H'の中で共役}, all subgroups isomorphic to $A_4$ in $G$ that fix $\infty$ in $B$ have the form
$\left\langle\begin{psmallmatrix} \alpha & y \\ 0 & \alpha^2 \end{psmallmatrix}, \begin{psmallmatrix} 1 & x \\ 0 & 1 \end{psmallmatrix}\right\rangle$ where $\alpha \in \{\omega,\omega^2\}$, $x \in \GF(2^n)^\times$, $y \in \GF(2^n)$. In particular, if we let $H'$ be a subgroup isomorphic to $A_4$ in $G$ which contains $\begin{psmallmatrix} 1 & 1 \\ 0 & 1 \end{psmallmatrix}$ and fixes $\infty$ in $B$, then $H'=\left\langle\begin{psmallmatrix} 1 & 1 \\ 0 & 1 \end{psmallmatrix},\begin{psmallmatrix} \alpha & y \\ 0 & \alpha^2 \end{psmallmatrix}\right\rangle=$
\[\left\{\begin{psmallmatrix} 1 & 0 \\ 0 & 1 \end{psmallmatrix}, 
\begin{psmallmatrix} 1 & 1 \\ 0 & 1 \end{psmallmatrix}, 
\begin{psmallmatrix} 1 & \omega \\ 0 & 1 \end{psmallmatrix}, 
\begin{psmallmatrix} 1 & \omega^2 \\ 0 & 1 \end{psmallmatrix}, 
\begin{psmallmatrix} \omega^2 & y \\ 0 & \omega \end{psmallmatrix}, 
\begin{psmallmatrix} \omega^2 & y+1 \\ 0 & \omega \end{psmallmatrix}, 
\begin{psmallmatrix} \omega^2 & y+\omega \\ 0 & \omega \end{psmallmatrix}, 
\begin{psmallmatrix} \omega^2 & y+\omega^2 \\ 0 & \omega \end{psmallmatrix}, 
\begin{psmallmatrix} \omega & y \\ 0 & \omega^2 \end{psmallmatrix}, 
\begin{psmallmatrix} \omega & y+1 \\ 0 & \omega^2 \end{psmallmatrix}, 
\begin{psmallmatrix} \omega & y+\omega \\ 0 & \omega^2 \end{psmallmatrix}, 
\begin{psmallmatrix} \omega & y+\omega^2 \\ 0 & \omega^2 \end{psmallmatrix}\right\}.\]
Therefore $H_{1,a}, H_{2,a}, H_{3,a}$ cannot be subgroups of $H'$ since $a\notin \{w,w^{2}\}$. Since $H$ fixes $\infty$ in $B$, any $g \in N_G(H)$ must also fix $\infty$, and therefore $g$ has the form $g=\begin{psmallmatrix} g_1 & g_2 \\ 0 & g_1^{-1} \end{psmallmatrix}$ where $g_1 \in \GF(2^n)^\times$ and $g_2 \in \GF(2^n)$. Then we have
$g\begin{psmallmatrix} 1 & 1 \\ 0 & 1 \end{psmallmatrix}g^{-1}=\begin{psmallmatrix} 1 & g_1^2 \\ 0 & 1 \end{psmallmatrix}$ and
$g\begin{psmallmatrix} 1 & a \\ 0 & 1 \end{psmallmatrix}g^{-1}=\begin{psmallmatrix} 1 & g_1^2a \\ 0 & 1 \end{psmallmatrix}$. Therefore, 
$\left\langle\begin{psmallmatrix} 1 & g_1^2 \\ 0 & 1 \end{psmallmatrix},\begin{psmallmatrix} 1 & g_1^2a \\ 0 & 1 \end{psmallmatrix}\right\rangle = \left\langle\begin{psmallmatrix} 1 & 1 \\ 0 & 1 \end{psmallmatrix},\begin{psmallmatrix} 1 & a \\ 0 & 1 \end{psmallmatrix}\right\rangle
= \left\{\begin{psmallmatrix} 1 & 0 \\ 0 & 1 \end{psmallmatrix},\begin{psmallmatrix} 1 & 1 \\ 0 & 1 \end{psmallmatrix},\begin{psmallmatrix} 1 & a \\ 0 & 1 \end{psmallmatrix},\begin{psmallmatrix} 1 & a+1 \\ 0 & 1 \end{psmallmatrix}\right\}$.
For $a \notin \{\omega,\omega^2\}$, this implies $g_1 = 1$. Therefore
$N_G(H) = \left\{\begin{psmallmatrix} 1 & g_2 \\ 0 & 1 \end{psmallmatrix} \mid g_2 \in \GF(2^n)\right\}$.
From the form of $N_G(H)$, we obtain $|N_G(H)| = 2^n$. Since the stabilizer of $H$ under conjugation by $G$ is equal to $N_G(H)$, by the orbit-stabilizer theorem:
\[|\Orbit(H)| = \frac{|G|}{|N_G(H)|} = \frac{2^n(4^n-1)}{2^n} = 4^n-1.\]
Let $\mathcal{B}'_{\mathbb{Z}_{2}^{2}}=\{B\in \binom{X}{13} \mid G_B \simeq \mathbb{Z}_{2}^{2}, a\notin \{0,1,\omega,\omega^2\}\}$. By Lemma~\ref{lem:5.16 Z22完全代表系}, we have $\frac{1}{3}\cdot \frac{2^n-4}{2}$ possibilities with $|\Orbit(H)|=4^n-1$ choices for each $H$. Since we have shown $\mathcal{B}_{A_{4}}^{g_{1}}\not \subset \mathcal{B}_{\mathbb{Z}_{2}^{2}}$, we now count the possible choices for $B$ for each fixed $H$. Any set $B$ with $G_B \simeq \mathbb{Z}_2^2$ must have the form
\[B = \{x,x+1,x+a,x+a+1, y,y+1,y+a,y+a+1,z,z+1,z+a,z+a+1,\infty\},\]
where $x,y,z \in \GF(2^n)$ and all elements in $B$ are distinct. When constructing such a set, we have $2^n$ possibilities for $x$, but since choosing any element from $\{x,x+1,x+a,x+a+1\}$ generates the same subset, we effectively have $\frac{2^n}{4}$ choices. Similarly, we have $\frac{2^n-4}{4}$ choices for the second block started with $y$, and $\frac{2^n-8}{4}$ choices for the third block started with $z$. Furthermore, since the ordering of these three blocks does not affect the resulting set $B$, we must divide by $3!$. Therefore, for each fixed $H$, the number of possible choices for $B$ is $\frac{2^n}{4} \cdot \frac{2^n-4}{4} \cdot \frac{2^n-8}{4} \cdot \frac{1}{3!}$. Hence:
\[|\mathcal{B}'_{\mathbb{Z}_{2}^{2}}| = \frac{1}{3}\cdot \frac{2^n-4}{2}\cdot (4^n-1) \cdot \frac{2^n}{4} \cdot \frac{2^n-4}{4} \cdot \frac{2^n-8}{4} \cdot \frac{1}{3!}.\]

(B): We next consider the case $a \in \{\omega,\omega^2\}$. Let $H = \left\langle\begin{psmallmatrix} 1 & 1 \\ 0 & 1 \end{psmallmatrix},\begin{psmallmatrix} 1 & \omega \\ 0 & 1 \end{psmallmatrix}\right\rangle = \left\{\begin{psmallmatrix} 1 & 0 \\ 0 & 1 \end{psmallmatrix}, \begin{psmallmatrix} 1 & 1 \\ 0 & 1 \end{psmallmatrix}, \begin{psmallmatrix} 1 & \omega \\ 0 & 1 \end{psmallmatrix}, \begin{psmallmatrix} 1 & \omega^2 \\ 0 & 1 \end{psmallmatrix}\right\}$. All subgroups isomorphic to $A_4$ in $G$ that fix $\infty$ in $B$ have the form $\left\langle\begin{psmallmatrix} \alpha & y \\ 0 & \alpha^2 \end{psmallmatrix}, \begin{psmallmatrix} 1 & x \\ 0 & 1 \end{psmallmatrix}\right\rangle$ where $\alpha \in \{\omega,\omega^2\}$, $x \in \GF(2^n)^\times$, $y \in \GF(2^n)$. Therefore, if $H'$ is a subgroup isomorphic to $A_4$ which contains $\begin{psmallmatrix} 1 & 1 \\ 0 & 1 \end{psmallmatrix}$ and fixes $\infty$ in $B$, then $H'=\left\langle\begin{psmallmatrix} 1 & 1 \\ 0 & 1 \end{psmallmatrix},\begin{psmallmatrix} \alpha & y \\ 0 & \alpha^2 \end{psmallmatrix}\right\rangle$. Let $h\in N_G(H)$ be of the form $h=\begin{psmallmatrix} h_1 & h_2 \\ 0 & h_1^{-1} \end{psmallmatrix}$. Then
$\left\langle\begin{psmallmatrix} 1 & h_1^2 \\ 0 & 1 \end{psmallmatrix},\begin{psmallmatrix} 1 & h_1^2\omega \\ 0 & 1 \end{psmallmatrix}\right\rangle = \left\langle\begin{psmallmatrix} 1 & 1 \\ 0 & 1 \end{psmallmatrix},\begin{psmallmatrix} 1 & \omega \\ 0 & 1 \end{psmallmatrix}\right\rangle
= \left\{\begin{psmallmatrix} 1 & 0 \\ 0 & 1 \end{psmallmatrix},\begin{psmallmatrix} 1 & 1 \\ 0 & 1 \end{psmallmatrix},\begin{psmallmatrix} 1 & \omega \\ 0 & 1 \end{psmallmatrix},\begin{psmallmatrix} 1 & \omega^2 \\ 0 & 1 \end{psmallmatrix}\right\}.$ This implies $h_1^2\in \{1,\omega,\omega^2\}$, which gives $h_1\in \{1,\omega,\omega^2\}$. Therefore
$N_G(H) = \left\{\begin{psmallmatrix} 1 & b \\ 0 & 1 \end{psmallmatrix},\begin{psmallmatrix} \omega & b \\ 0 & \omega^2 \end{psmallmatrix},\begin{psmallmatrix} \omega^2 & b \\ 0 & \omega \end{psmallmatrix} \mid b \in \GF(2^n)\right\}$. Thus $|N_G(H)| = 3\cdot 2^n$, and by the orbit-stabilizer theorem:
\[|\Orbit(H)| = \frac{|G|}{|N_G(H)|} = \frac{2^n(4^n-1)}{3\cdot 2^n} = \frac{4^n-1}{3}.\]

Let $\mathcal{B}''_{\mathbb{Z}_{2}^{2}}=\{B\in \binom{X}{13} \mid G_B \simeq \mathbb{Z}_{2}^{2}, a\in \{\omega,\omega^2\}\}$. For each $H$, the number of possible choices for $B$ is $\frac{2^n(2^n-4)(2^n-8)}{4^3 \cdot 3!} - |\mathcal{B}_{A_4}^{g_1}|$, where $|\mathcal{B}_{A_4}^{g_1}| = \#\{B\in \binom{X}{13} \mid G_B \simeq A_{4}, g_1(B)=B\}$ where $g_1=\begin{psmallmatrix} 1 & 1 \\ 0 & 1 \end{psmallmatrix}$. By Lemma~\ref{lem:5.10 BA4g1}, we have $|\mathcal{B}_{A_4}^{g_1}| = \frac{2^n(2^n-4)}{48}$. Therefore:
\[|\mathcal{B}''_{\mathbb{Z}_{2}^{2}}| = |\Orbit(H)|\cdot \left(\frac{2^n(2^n-4)(2^n-8)}{4^3 \cdot 3!}-\frac{2^n(2^n-4)}{48}\right) = \frac{4^n-1}{3}\left(\frac{2^n(2^n-4)(2^n-8)}{4^3 \cdot 3!}-\frac{2^n(2^n-4)}{48}\right).\]
Hence, for even $n$:
\[|\mathcal{B}_{\mathbb{Z}_{2}^{2}}| = |\mathcal{B}'_{\mathbb{Z}_{2}^{2}}| + |\mathcal{B}''_{\mathbb{Z}_{2}^{2}}| = \frac{2^n(4^n-1)(2^n-2)(2^n-4)(2^n-8)}{2304}.\]

For odd $n$, since there are no primitive cube roots in $\GF(2^n)$, we can apply the same argument as case (A) above. Note that the number of subgroups isomorphic to $\mathbb{Z}_{2}^{2}$ is $\frac{1}{3}\cdot \frac{2^{n}-2}{2}$. Therefore, $|\mathcal{B}_{\mathbb{Z}_{2}^{2}}|=\frac{2^n-2}{2} \cdot(4^{n}-1)\cdot  \frac{2^n(2^n-4)(2^n-8)}{4^3 \cdot 3!} = \frac{2^n(4^n-1)(2^n-2)(2^n-4)(2^n-8)}{2304}$.
\end{proof}

\begin{lem}\label{lem:5.18 BZ22g1}
Let 
$\mathcal{B}_{\mathbb{Z}_{2}^{2}}^{g_1} = \{B\in \binom{X}{13} \mid G_B \simeq \mathbb{Z}_{2}^{2}, g_1(B)=B\}$ where $g_1 = \begin{psmallmatrix} 1 & 1 \\ 0 & 1 \end{psmallmatrix}$. Then
\[|\mathcal{B}_{\mathbb{Z}_2^2}^{g_1}| = \begin{cases}
\dfrac{2^{2n}(2^n-4)(2^n-10)}{768} & \text{if}\,\,n\,\,\text{is even},  \\
\dfrac{2^n(2^n-2)(2^n-4)(2^n-8)}{768} & \text{if}\,\,n\,\,\text{is odd}. 
\end{cases}\]
\end{lem}
\begin{proof}
First, we consider the case where $n$ is even. By Lemma~\ref{lem:5.15 無限を固定するZ22は1101が入る形に共役である}, any subgroup isomorphic to $\mathbb{Z}_{2}^{2}$ fixing $\infty$ and containing $g_1=\begin{psmallmatrix} 1 & 1 \\ 0 & 1 \end{psmallmatrix}$ must be of the form $H \coloneq \left\langle\begin{psmallmatrix} 1 & 1 \\ 0 & 1 \end{psmallmatrix},\begin{psmallmatrix} 1 & a \\ 0 & 1 \end{psmallmatrix}\right\rangle$ for some $a\in \GF(2^n)\setminus\{0,1\}$. Since replacing $a$ by $a+1$ yields the same group $H$, we obtain $\frac{2^n-2}{2}$ possibilities for the group $H$.

Let $\mathcal{B}^{g_1}_{\mathbb{Z}_{2}^{2}}{'} = \{B\subset \binom{X}{13} \mid G_B \simeq \mathbb{Z}_{2}^{2}, g_1(B)=B, a\notin \{\omega,\omega^2\}\}$ where $g_1=\begin{psmallmatrix} 1 & 1 \\ 0 & 1 \end{psmallmatrix}$ and 
$\mathcal{B}^{g_1}_{\mathbb{Z}_{2}^{2}}{''} = \{B\subset \binom{X}{13} \mid G_B \simeq \mathbb{Z}_{2}^{2}, g_1(B)=B, a\in \{\omega,\omega^2\}\}$  where $g_1=\begin{psmallmatrix} 1 & 1 \\ 0 & 1 \end{psmallmatrix}$. As shown in Lemma~\ref{lem:5.17 BZ22}, when $a\notin \{\omega,\omega^2\}$, $H$ cannot be a subgroup of $A_4$ which fixes $\infty$ in $B$, but when $a\in \{\omega,\omega^2\}$, $H$ can be a subgroup. Following Lemma~\ref{lem:5.17 BZ22}, we consider two cases: (A) $a\notin \{\omega,\omega^2\}$ and (B) $a\in \{\omega,\omega^2\}$.

(A): For the total number of possible groups $H$, we have shown there are $\frac{2^n-2}{2}$ possibilities. Since we require $a\notin \{\omega,\omega^2\}$, we must subtract one possibility, giving $\frac{2^n-4}{2}$ possibilities for $H$. By the same argument as in Lemma~\ref{lem:5.17 BZ22}, we obtain:
\[|\mathcal{B}^{g_1}_{\mathbb{Z}_{2}^{2}}{'}| = \frac{2^n-4}{2}\cdot \frac{2^n(2^n-4)(2^n-8)}{4^3 \cdot 3!}.\]

(B): When $a\in \{\omega,\omega^2\}$, we have $H = \left\langle\begin{psmallmatrix} 1 & 1 \\ 0 & 1 \end{psmallmatrix},\begin{psmallmatrix} 1 & \omega \\ 0 & 1 \end{psmallmatrix}\right\rangle$, so there is exactly one possibility for $H$. By Lemma~\ref{lem:5.10 BA4g1}, we must subtract $|\mathcal{B}_{A_4}^{g_1}| = \frac{2^n(2^n-4)}{48}$. Therefore:
\[|\mathcal{B}^{g_1}_{\mathbb{Z}_{2}^{2}}{''}| = \frac{2^n(2^n-4)(2^n-8)}{4^3 \cdot 3!}-\frac{2^n(2^n-4)}{48}.\]
Since $|\mathcal{B}_{\mathbb{Z}_{2}^{2}}^{g_1}| = |\mathcal{B}^{g_1}_{\mathbb{Z}_{2}^{2}}{'}| + |\mathcal{B}^{g_1}_{\mathbb{Z}_{2}^{2}}{''}| $, we obtain:
\[|\mathcal{B}_{\mathbb{Z}_2^2}^{g_1}| = \frac{2^n(2^n-2)(2^n-4)(2^n-8)}{768} - \frac{2^n(2^n-4)}{48}=\frac{2^{2n}(2^n-4)(2^n-10)}{768}.\]
For odd $n$, since there are no primitive cube roots in $\GF(2^n)$, we can apply the same argument as case (A) above. Note that we obtain $\frac{2^n-2}{2}$ possibilities for the group $H$. Therefore:
\[|\mathcal{B}_{\mathbb{Z}_2^2}^{g_1}| = \frac{2^n-2}{2} \cdot \frac{2^n(2^n-4)(2^n-8)}{4^3 \cdot 3!} = \frac{2^n(2^n-2)(2^n-4)(2^n-8)}{768}.\]
\end{proof}

\begin{prop}\label{prop:5.19 m429}
\[
m_{429} = \begin{cases}
\dfrac{2^{n}(2^n-4)(2^n-10)}{576} & \text{if}\,\,n\,\,\text{is even},   \\
\dfrac{(2^n-2)(2^n-4)(2^n-8)}{576} & \text{if}\,\,n\,\,\text{is odd}.
\end{cases}
\]
\end{prop}

\begin{proof}
Note that $|O_2|=4^n-1$. Hence by Cauchy-Frobenius-Burnside lemma, 
\[m_{429} = \frac{1}{|G|}\left(|\mathcal{B}_{\mathbb{Z}_2^2}|+(4^n-1)|\mathcal{B}_{\mathbb{Z}_2^2}^{g_1}|\right),\]
where $g_1 = \begin{psmallmatrix} 1 & 1 \\ 0 & 1 \end{psmallmatrix}$. The claim follows by substituting the values from Lemmas~\ref{lem:5.17 BZ22} and \ref{lem:5.18 BZ22g1}.
\end{proof}

\subsection{Calculation of $m_{572}$}

\begin{remark}
When $G_B \simeq \mathbb{Z}_{3}$, the corresponding value $\lambda_B = 572$ represents the parameter of the $3$-$(2^n+1,13,572)$ design formed by the orbit $G(B)$. We aim to determine $m_{572}$, which represents the total number of orbits forming such designs.
\end{remark}

\begin{lem}\label{lem:5.21 Z3は共役}
All subgroups isomorphic to $\mathbb{Z}_3$ in $G$ are conjugate to each other.
\end{lem}

\begin{proof}
By Table~\ref{Table:1 置換指標}, all elements of order $3$ in $G$ are conjugate. Since any subgroup isomorphic to $\mathbb{Z}_3$ is cyclic and uniquely determined by its generator of order $3$, the result follows immediately.
\end{proof}

\begin{prop}\label{prop:5.22 m572}
\[
m_{572} = \begin{cases}
\dfrac{(2^n-1)(2^n-4)(2^n-16)}{648} & \text{if}\,\,n\,\,\text{is even}, \\
0 & \text{if}\,\,n\,\,\text{is odd}.
\end{cases}
\]
\end{prop}

\begin{proof}
When $n$ is odd, there is no element of order $3$, thus $m_{572}=0$. Therefore, we consider the case where $n$ is even. Let $\mathcal{B}_{\mathbb{Z}_3} \coloneq \{B\in \binom{X}{13} \mid G_B \simeq \mathbb{Z}_{3}\}$ and $\mathcal{B}_{\mathbb{Z}_3}^{g_{2}} \coloneq \{B\in \binom{X}{13} \mid G_B \simeq \mathbb{Z}_{3}, g_{2}(B)=B\}$ where $g_{2}= \begin{psmallmatrix} \omega^2 & 0 \\ 0 & \omega \end{psmallmatrix}$. We have $|O_3| = 2^n(2^n+1)$, and by Cauchy-Frobenius-Burnside lemma, we obtain
\begin{equation}\label{eq:5.2 m572の定式化}
m_{572} = \frac{1}{|G|}\left(|\mathcal{B}_{\mathbb{Z}_3}|+2^n(2^n+1)|\mathcal{B}_{\mathbb{Z}_3}^{g_2}|\right).  
\end{equation}
Let $\mathcal{H} \coloneq \{H \subset G \mid H \simeq \mathbb{Z}_3\}$. By Lemma~\ref{lem:5.21 Z3は共役}, all subgroups isomorphic to $\mathbb{Z}_3$ are conjugate. Since $|O_3| = 2^n(2^n-1)$ and each $H$ contains two elements of order 3, we have $|\mathcal{H}| = \frac{2^n(2^n-1)}{2}$. 

First, we determine $|\mathcal{B}_{\mathbb{Z}_3}|$. Since all subgroups in $\mathcal{H}$ are conjugate, let $H \coloneq \left\langle\begin{psmallmatrix} \omega^2 & 0 \\ 0 & \omega \end{psmallmatrix}\right\rangle \subset \mathcal{H}$. Since $H$ fixes $0$ and $\infty$, each possibility for $B$ must contain one of these fixed points and takes the form:
\[
B = \{a,a\omega,a\omega^2,b,b\omega,b\omega^2,c,c\omega,c\omega^2,d,d\omega,d\omega^2\}\cup\{0\}\text{ or }\{\infty\}.
\]
Since we have two choices for the fixed point ($0$ or $\infty$), and for each fixed point there are $\frac{2^n-1}{3} \cdot \frac{2^n-4}{3} \cdot \frac{2^n-7}{3} \cdot \frac{2^n-10}{3} \cdot \frac{1}{4!}$ possibilities for $B$, the total number of possibilities is:
\[
2\cdot \frac{2^n-1}{3} \cdot \frac{2^n-4}{3} \cdot \frac{2^n-7}{3} \cdot \frac{2^n-10}{3} \cdot \frac{1}{4!}=\frac{(2^{n}-1)(2^{n}-4)(2^{n}-7)(2^{n}-10)}{972}.
\]
However, this count includes cases where $B \in \mathcal{B}_{A_4}^{g_2}$. By Lemma~\ref{lem:5.11 BA4g2}, we need to subtract $|\mathcal{B}_{A_4}^{g_2}| = \frac{(2^n-1)(2^n-4)}{18}$. Therefore, the exact number of possible $B$ for a fixed $H$ is $\frac{(2^n-1)(2^n-4)(2^n-7)(2^n-10)}{972}-\frac{(2^n-1)(2^n-4)}{18}$. Since we can choose $H$ in $\frac{2^n(2^n-1)}{2}$ ways, we have:
\[
|\mathcal{B}_{\mathbb{Z}_3}| = \frac{2^n(2^n-1)}{2}\left(\frac{(2^n-1)(2^n-4)(2^n-7)(2^n-10)}{972}-\frac{(2^n-1)(2^n-4)}{18}\right).
\]
Recall that we defined $\mathcal{B}_{\mathbb{Z}_3}^{g_2} \coloneq \{B\in \binom{X}{13} \mid G_B \simeq \mathbb{Z}_{3}, g_{2}(B)=B\}$ where $g_{2}= \begin{psmallmatrix} \omega^2 & 0 \\ 0 & \omega \end{psmallmatrix}$. Since $g_2 \in H$, there is clearly only one possibility for such an $H$. For this unique $H$, by the same argument as above, we have:
\[
|\mathcal{B}_{\mathbb{Z}_3}^{g_2}| = \frac{(2^n-1)(2^n-4)(2^n-7)(2^n-10)}{972}-\frac{(2^n-1)(2^n-4)}{18}.
\]
Substituting these expressions into $(\ref{eq:5.2 m572の定式化})$ yields $m_{572} = \frac{(2^n-1)(2^n-4)(2^n-16)}{648}$.
\end{proof}

\subsection{Calculation of $m_{858},m_{1716}$}

\begin{remark}
When $G_B \simeq \mathbb{Z}_2$, the corresponding value $\lambda_B = 858$ represents the parameter of the $3$-$(2^n+1,13,858)$ design formed by the orbit $G(B)$. Similarly, when $G_B \simeq \text{Id}$, we have $\lambda_B = 1716$.

To determine $m_{858}$ and $m_{1716}$, we will use a different approach than the direct application of the Cauchy-Frobenius-Burnside lemma used for $m_{143}$, $m_{429}$, and $m_{572}$. A direct approach would involve considerably more complex case analysis due to the larger number of possible configurations.

Instead, we solve a system of equations using Lemma~\ref{lem:5.1 連立方程式}. Since the values of $m_{78}$ (established in Proposition~\ref{prop:3.2 m78=1}), $m_{66}$, $m_{143}$, $m_{429}$, and $m_{572}$ (determined in previous subsections) are now known, we can substitute these values into the equations (A)' and (B)' to obtain a system of two equations with $m_{858}$ and $m_{1716}$ as the only unknowns.
\end{remark}

\begin{prop}\label{prop:5.24 m858}
Let $q=2^n$. Define the value of $W_n$ as follows:
\[W_n = \frac{1}{23040}q(q-2)(q-4)(q-8)(q-16).\]
The value of $m_{858}$ is then determined as follows:
\[m_{858} = \begin{cases}
W_n - 1, &\text{if } n \equiv 6,10,12,18,20,24,36,40,42,48,50,54 \pmod{60}, \\
W_n - 2, &\text{if } n \equiv 0,30 \pmod{60}, \\
W_n, & \text{otherwise}.
\end{cases}\]
\end{prop}
 
\begin{prop}\label{prop:5.25 m1716}
Remark that the values of $X_{1,n}$ and $X_{2,n}$ are defined in Proposition~\ref{prop:4.3 軌道の総数N13} as follows:
\[X_{1,n}=\frac{1}{6227020800}(q-2)(q-3)(q-4)(q-5)(q-6)(q-7)(q-8)(q-9)(q-10)(q-11),\]
\[X_{2,n}=\frac{1}{46080}(q-2)(q-4)(q-6)(q-8)(q-10).\]
Let $X'_n = X_{1,n}-X_{2,n} + \frac{1}{1152}(q-2)(q-4)(q-8)$ and $Y_n = -\frac{1}{1944}(q-1)(q-4)(q-16)$. Then $m_{1716}$ is determined as follows:
\[m_{1716} = \begin{cases}
X'_n, & \text{if}\,\,n\,\,\text{is odd}, \\
X'_n + Y_n, &\text{if } n \equiv 2,4,8,14,16,22,26,28,32,34,38,44,46,52,56,58 \pmod{60}, \\
X'_n + Y_n + \frac{5}{11}, &\text{if } n \equiv 10,20,40,50 \pmod{60}, \\
X'_n + Y_n + \frac{6}{13}, &\text{if } n \equiv 6,12,18,24,36,42,48,54 \pmod{60}, \\
X'_n + Y_n + \frac{131}{143}, &\text{if } n \equiv 0,30 \pmod{60}, (n>0).
\end{cases}\]
\end{prop}

\section{Conclusion}
\label{sec:conclusion}

In this section, we summarize the results obtained in the previous sections and state our main theorem. Previously, it was known that for any block $B \in \binom{X}{13}$, the orbit $G(B)$ under the action of $G$ forms a $3$-design.

Therefore, the main objective of this paper was to determine all possible $3$-$(2^n+1,13,\lambda)$ designs admitting $\PSL(2,2^{n})$ as an automorphism group. To accomplish this, we needed to determine the parameter $\lambda_B$, which depends completely on $|G_B|$ as shown in equation~(\ref{eq:2.2 ラムダとG_B}), and the total number of orbits $m_{\lambda_B}$ that form 3-$(2^n+1,13,\lambda_B)$ designs for each realizable value of $\lambda_B$.

Consequently, in $\S 3$, by examining the possible subgroup structures of $\PSL(2,2^{n})$, we established Table~\ref{table:4 GB候補2}, from which we determined that
\[|G_B|\in\{1,2,3,4,12,22,26\}, \quad \lambda_{B}\in \{66,78,143,429,572,858,1716\}.\]

Moreover, in $\S 3$, we also determined the values of $m_{66}$ and $m_{78}$ in Propositions~\ref{prop:3.5 m66=1} and~\ref{prop:3.2 m78=1}, respectively, as follows:
\[
m_{66} = \begin{cases}
1 & \text{if } n \equiv 0 \pmod{6}, \\
0 & \text{otherwise},
\end{cases}
\quad
m_{78} = \begin{cases}
1 & \text{if }n \equiv 0 \pmod{10},\\
0 & \text{otherwise}.
\end{cases}
\]

In $\S 4$, we computed the total number of orbits of 13-element subsets, denoted by $N_{13}$. Let $q=2^n$ and define:
\begin{equation*}
\begin{cases}
X_{1,n}=\frac{1}{6227020800}(q-2)(q-3)(q-4)(q-5)(q-6)(q-7)(q-8)(q-9)(q-10)(q-11),\\
X_{2,n}=\frac{1}{46080}(q-2)(q-4)(q-6)(q-8)(q-10),\\
X_{3,n}=\frac{1}{972}(q-4)(q-7)(q-10),\\
X_{n}=X_{1,n}+X_{2,n}+X_{3,n}.
\end{cases}
\end{equation*}

According to Proposition~\ref{prop:4.3 軌道の総数N13}, the value of $N_{13}$ is determined as follows:
\[
N_{13} =
\begin{cases}
X_n, & \text{if } n \equiv 2,4,8,14,16,22,26,28,32,34,38,44,46,52,56,58 \pmod{60}, \\
X_n + \frac{5}{11}, & \text{if } n \equiv 10,20,40,50 \pmod{60}, \\
X_n + \frac{6}{13}, & \text{if } n \equiv 6,12,18,24,36,42,48,54 \pmod{60}, \\
X_n + \frac{131}{143}, & \text{if } n \equiv 0,30 \pmod{60}, \, n > 0, \\
X_{1,n} + X_{2,n}, & \text{if}\,\,n\,\,\text{is odd}.
\end{cases}
\]

In $\S 5$, we determined $m_{143}$, $m_{429}$, and $m_{572}$ through Propositions~\ref{prop:5.12 m143}, \ref{prop:5.19 m429}, and \ref{prop:5.22 m572} respectively, as follows:
\[
m_{143} = \begin{cases}
\frac{2^n-4}{12} & \text{if}\,\,n\,\,\text{is even}, \\
0 & \text{if}\,\,n\,\,\text{is odd}, 
\end{cases}
\quad
m_{429} = \begin{cases}
\frac{2^{n}(2^n-4)(2^n-10)}{576} & \text{if}\,\,n\,\,\text{is even},   \\
\frac{(2^n-2)(2^n-4)(2^n-8)}{576} & \text{if}\,\,n\,\,\text{is odd}, 
\end{cases}
\]
\[
m_{572} = \begin{cases}
\frac{(2^n-1)(2^n-4)(2^n-16)}{648} & \text{if}\,\,n\,\,\text{is even},   \\
0 & \text{if}\,\,n\,\,\text{is odd}.
\end{cases}
\]

Furthermore, by substituting the values of $m_{66}$, $m_{78}$, $m_{143}$, $m_{429}$, $m_{572}$, and $N_{13}$ into the system of equations (A)' and (B)' from Lemma~\ref{lem:5.1 連立方程式}, we were able to determine $m_{858}$ and $m_{1716}$. Let us define the following parameters:
\begin{equation*}
\begin{cases}
X'_n = X_{1,n}-X_{2,n} + \frac{1}{1152}(q-2)(q-4)(q-8),\\
Y_n = -\frac{1}{1944}(q-1)(q-4)(q-16),\\
W_n = \frac{1}{23040}q(q-2)(q-4)(q-8)(q-16).
\end{cases}
\end{equation*}

The values of $m_{858}$ and $m_{1716}$ are then determined through Propositions~\ref{prop:5.24 m858} and \ref{prop:5.25 m1716} respectively, as follows:
\[m_{858} = \begin{cases}
W_n - 1, & \text{if } n \equiv 6,10,12,18,20,24,36,40,42,48,50,54 \pmod{60}, \\
W_n - 2, & \text{if } n \equiv 0,30 \pmod{60}, \\
W_n, & \text{otherwise},
\end{cases}\]

and

\[m_{1716} = \begin{cases}
X'_n, & \text{if}\,\,n\,\,\text{is odd}, \\
X'_n + Y_n, & \text{if }n \equiv 2,4,8,14,16,22,26,28,32,34,38,44,46,52,56,58 \pmod{60}, \\
X'_n + Y_n + \frac{5}{11}, & \text{if } n \equiv 10,20,40,50 \pmod{60}, \\
X'_n + Y_n + \frac{6}{13}, & \text{if } n \equiv 6,12,18,24,36,42,48,54 \pmod{60}, \\
X'_n + Y_n + \frac{131}{143}, & \text{if } n \equiv 0,30 \pmod{60}, (n>0).
\end{cases}\]

Since the union of orbits also forms a $3$-design, we were able to establish our main theorem regarding the feasible values of $\lambda_B$.

\begin{thm}
\label{thm:main}
There exists a simple $3$-$(2^n+1,13,\lambda)$ design with $\PSL(2,2^{n})$ as an automorphism group if and only if 
\[\lambda=66i_{66}+78i_{78}+143i_{143}+429i_{429}+572i_{572}+858i_{858}+1716i_{1716},\]
where $0\le i_{\lambda_{B}}\le m_{\lambda_{B}}$ and $\lambda_{B}\in \{66,78,143,429,572,858,1716\}$, with $m_{\lambda_{B}}$ being the corresponding values as defined above.
\end{thm}

\begin{remark}
In previous research concerning the computation of $m_{\lambda_B}$ for block sizes $k=4,5,6,7,8,9$, the approach involved normalizing the sets to include $\{0,1,\infty\}$ and then dividing the number of possible sets $B$ by the number of sets with the same form contained in $G(B)$. For instance, as shown in Lemma~4.1 for block size $k=7$ \cite{6}, the method involved counting the number of sets $B$ and dividing by $2$.
However, as the block size increases, it becomes increasingly difficult to enumerate both the possible sets $B$ and the appropriate divisor. By applying the Cauchy-Frobenius-Burnside lemma to compute $m_{\lambda_B}$, we suggest that this approach could be effective not only for the yet-unstudied cases of $k=10,11,12$ but also for block sizes larger than $13$.
\end{remark}

\section*{Acknowledgements}
The authors would like to express their sincere gratitude to Professor Akihiro Munemasa for suggesting the research direction that motivated this work. We are deeply indebted to Professor Koichi Betsumiya for his constant guidance and supervision throughout this research.

We also wish to thank Professor Akihide Hanaki and Professor Yuta Kozakai for carefully reading our manuscript and providing valuable feedback that significantly improved the quality of this paper.

The support and encouragement from all these mentors have been essential to the completion of this work.

\end{document}